\documentclass[10pt]{article}%

\usepackage[utf8]{inputenc}

\usepackage{amsmath}

\usepackage{amsfonts}
\usepackage{mathrsfs}
\usepackage{dsfont}
\usepackage{amssymb, color}
\usepackage[linkcolor=black,anchorcolor=black,citecolor=black]{hyperref}
\usepackage{graphicx}
\numberwithin{equation}{section}
\usepackage[body={15.5cm,21cm}, top=3cm]{geometry}%
\providecommand{\U}[1]{\protect\rule{.1in}{.1in}}
\providecommand{\U}[1]{\protect \rule{.1in}{.1in}}
\newtheorem{theorem}{Theorem}[section]

\newtheorem{lemma}[theorem]{Lemma}

\newtheorem{proposition}[theorem]{Proposition}
\newtheorem{remark}[theorem]{Remark}

\newenvironment{proof}[1][Proof]{\noindent \textbf{#1.} }{\  \rule{0.5em}{0.5em}}
\DeclareMathOperator*{\esssup}{ess\,sup}
\DeclareMathOperator*{\essinf}{ess\,inf}

\def \E{\mathsf{E}}

\def \P{\mathsf{P}}
\def \R{\mathbb{R}}

\def\d{\mathrm{d}}

\begin{document}
	\title{Optimal Consumption with Intertemporal Substitution under Knightian Uncertainty}
	\author{ Giorgio Ferrari\thanks{Center for Mathematical Economics, Bielefeld University, giorgio.ferrari@uni-bielefeld.de.}
		\and	Hanwu Li\thanks{Center for Mathematical Economics, Bielefeld University,		hanwu.li@uni-bielefeld.de.}
		\and Frank Riedel \thanks{Center for Mathematical Economics, Bielefeld University,		frank.riedel@uni-bielefeld.de.}}
\date{\today}
	\maketitle
	
\begin{abstract}
We study an intertemporal consumption and portfolio choice problem under Knightian uncertainty   in which agent's preferences exhibit local intertemporal substitution. We also allow for market frictions in the sense that the pricing functional is  nonlinear. We prove existence and uniqueness of the optimal consumption plan,  and we derive a set of sufficient first-order conditions for optimality. With the help of a backward equation, we are able to determine the structure of   optimal consumption plans. We  obtain explicit solutions  in a   stationary setting in which the  financial market has  different risk premia for short and long positions.
\end{abstract}





\section{Introduction}
\label{sec:intro}

In the seminal papers \cite{HH0, HHK}, a fundamental critique to the standard time-additive framework for optimal consumption problems was raised. A.\ Hindy, C.F.\ Huang and D.\ Kreps noticed that time-additive utility functionals, and more generally all utility functionals directly depending on the rate of consumption, are not robust with respect to slight shifts of consumption in time. To overcome such an undesired feature,  these authors  proposed to measure utility arising from consumption through the so-called level of satisfaction, a suitably  weighted average of past consumption\footnote{This represents only one of many possible ways in which to capture effects of intertemporal substitution of preferences and/or habit formation. We refer the reader to \cite{EnglezosK}, \cite{Riedel}, and \cite{Watson} and references therein.} With such a preference the agent measures her felicity by taking into account also the history of her consumption plan and considers consumption at adjacent times as similar alternatives. As a consequence, the agent consumes not necessarily in rates, but also in a singular manner or in gulps.

In \cite{HH} the authors consider an investment-consumption model for an agent who faces Hindy-Huang-Kreps (HHK) preferences and invests in a financial market without transaction costs. When the agent has an hyperbolic absolute risk aversion (HARA) instantaneous utility function and is active in a Black-Scholes market, explicit consumption and allocation choices have been derived. The case of general utility functions was then studied in \cite{Alvarez}, and later extended also to markets with jumps (and possibly transaction costs) in \cite{Benth1, Benth2, Benth3} and \cite{BR}. While \cite{Benth1, Benth2, Benth3} employ a dynamic programming approach to solve the considered optimal consumption problem, no Markovian structure is needed in \cite{BR}. The approach followed in \cite{BR} (see also \cite{BR00} for the deterministic case and \cite{BaKau} for a general convex-analytic treatment) indeed exploits the concave structure of the consumption problem and derives suitable necessary and sufficient Kuhn-Tucker-like first-order conditions for optimality (FOCs). Via such an alternative method, it shown that the optimal consumption plan is such to keep the level of satisfaction always above an endogenously determined minimal level solving a backward stochastic equation arising from the FOCs (see \cite{BE} for the study of such a class of equations).

In \cite{BR} (and actually also in all the other aforementioned references) there is the implicit assumption that, when picking the consumption plan maximizing the expected utility, the agent has complete knowledge about the probability distribution of all those random factors that affect her choice. Such an assumption can be clearly debatable in situations in which the agent is exposed to new phenomena affecting her preferences and for which there are not sufficient data available to conjecture the probability distribution with good confidence. Moreover, the characterization and construction of the optimal consumption plan is made in \cite{BR} under the assumption that the underlying financial market is frictionless. In this paper we consider the optimal consumption problem of an agent with preferences of HHK type, that faces Knightian uncertainty about the random factors affecting her utility from consumption, and that evaluates the costs of her consumption under a nonlinear expectation. In particular, such a nonlinear evaluation of the consumption's costs can be motivated by thinking that the agent finances her consumption plan in a financial market with frictions (see Section \ref{sec:explicit solution} for more details). We show that such a setting can be well encompassed by considering nonlinear expectations for the evaluation of utility/cost in the form of $g$-expectations. The theory of the latter was initiated by S.\ Peng in \cite{P97} and they are shown to be naturally related to backward stochastic differential equations (BSDEs) and to variational preferences. As a matter of fact, the $g$-expectation $\mathcal{E}^g[\,\cdot\,]$ is defined in terms of the first component of a solution to a BSDE with driver $g$, and can be represented as a variational expectation when $g$ is concave (in the second component of the BSDE's solution).

Following the arguments already developed in \cite{BR} in the linear case, we also show that existence of an optimal consumption plan can be obtained in our nonlinear setting by a suitable application of Koml\'os' theorem (see \cite{K67}) in the version of Y.\ Kabanov \cite{K99}. Moreover, if the felicity function $u$ is strictly concave in the level of satisfaction, also uniqueness of the optimal consumption plan can be established by exploiting the fact that the $g$-expectation satisfies the strict comparison property.

With the aim of providing a characterization of the optimal consumption plan, we then derive a set of sufficient FOCs for optimality. These are clearly different from those obtained in \cite{BR} where it is assumed that the underlying financial market is complete, and in fact degenerate to those if we take $g=0$. In our framework, the FOCs involve the so-called worst-case-scenarios, probability measure $\P_i$ with Girsanov kernel $\xi^i$ (with respect to the given reference probability measure $\P_0$) under which the lowest expected utility and the largest expected cost of consumption are realized. The sufficient FOCs can be used as a verification tool in order to check the actual optimality of a candidate optimal solution. In this sense our approach might then be seen a counterpart in our general not necessarily Markovian and nonlinear setting of the verification theorem usually employed in the study of Markovian optimal control problems addressed via the dynamic programming approach. Inspired by \cite{BR}, we then show that, for any given and fixed probability measures $\P_1$ and $\P_2$ under which expected utility and consumption's cost are evaluated, there exists a minimal level of satisfaction $L^{\P_1,\P_2}$ that solves a certain kind of backward equation studied in \cite{BE}. The consumption plan that tracks $L^{\P_1,\P_2}$ -- denoted by $C^{L^{\P_1,\P_2}}$ -- prescribes to consume just enough in order to keep the satisfaction's level above $L^{\P_1,\P_2}$ at all times. However, in order to find the optimal consumption a daunting fixed point problem has to be solved; in fact, $\P_1$ and $\P_2$ should be chosen in such a way that they realize the expected lowest utility largest expected cost of consumption.

In the generality of our framework, the complete study of such a fixed point problem would require a separate detailed analysis. This is why in a final section of our paper we consider a specific, yet relevant, setting in which a complete solution to the problem can be obtained by guessing, and then verifying, through the FOCs, the worst-case scenarios for utility and cost. We assume that the utility is of power-type, the time horizon is infinite, and that the set of multiple priors are constituted by measures that are equivalent to the given reference measure $\P_0$ and have Girsanov kernel belonging to suitable bounded intervals. We show that such a choice of the set of priors corresponds to the case in which the consumption plan is financed via investment in an underlying financial market in which the risk premia for short and long positions are different. Within this setting we provide an explicit form of the optimal consumption plan and we also determine the financing portfolio. Like in the classical Merton's problem, this prescribes to invest a fixed (in time) fraction of wealth in the risky asset traded in the financial market. Moreover, we observe that an increase in the risk and risk aversion leads the agent to invest less in the financial market. On the other hand, an increase of the discrepancy between the agent's beliefs about the evaluations of utility and cost from consumption has a non definite effect on the portfolio strategy which can either increase or decrease depending on the model's parameters.  Finally, we show that in the case in which the multiple priors for utility and cost have a common element, then the optimal minimal level of satisfaction is deterministic and, as a consequence, the agent will not invest in the risky asset at all.

In a series of remarks throughout the paper we also show how our results can be generalized to the case in which the evaluations of utility and cost of consumption are not made via $g$-expectations but through variational preferences induced by appropriate multiple priors and penalty functions.  In this framework, existence of an optimal consumption plan and the sufficient FOCs still hold.

The paper is organized as follows. In Section \ref{sec:problem} we formulate the utility maximization problem under Knightian uncertainty. We establish in Section \ref{sec:existence} existence and uniqueness of the optimal consumption plan, and we provide in Section \ref{FOCs} the sufficient FOCs for optimality. Section \ref{sec:structure} shows the time-consistency property of the optimal consumption plan and gives its general structure. The explicit solution in a specific setting is obtained in Section \ref{sec:explicit solution}
, while definition and properties of $g$-expectation as well as technical results are presented in \ref{app} and \ref{appB}, respectively.


\section{The Knightian intertemporal utility maximization problem}
\label{sec:problem}

Consider a filtered probability space $(\Omega,\mathcal{F}_T,(\mathcal{F}_t)_{t\in[0,T]},\P_0)$ satisfying the usual conditions of right-continuity and completeness and in which $B=\{B_t\}_{t\in[0,T]}$ is a $d$-dimensional Brownian motion. Throughout this paper, $\E[\,\cdot\,]$ will denote the expectation taken under the probability $\P_0$, and measurability properties (like progressive measurability or adaptedness) will be always with respect to $(\mathcal{F}_t)_{t\in[0,T]}$, as otherwise stated.

We aim at studying the optimal consumption choice of an agent facing Knightian ambiguity and whose preferences are of the Hindy-Huang-Kreps (HHK) type. We assume that the agent is ambiguity adverse and also that her consumption is financed via investment in a financial market which possibly exhibits frictions. As it will be clear in the sequel (see the setting of Section \ref{sec:explicit solution}), those agent's and market's characteristics can be well modeled through the use of nonlinear expectations; namely, via the $g$-expectation $\mathcal{E}^g[\,\cdot\,]$ initiated by S.\ Peng in \cite{P97}. As it is discussed in detail in \ref{app}, this is formally defined as the first component of the solution to a backward stochastic differential equation (BSDE) with driver $g:\Omega\times[0,T]\times \mathbb{R}^d\rightarrow\mathbb{R}$ satisfying the following requirements:
\begin{description}
	\item[(A1)] For any $z\in\mathbb{R}^d$, $g(\cdot,\cdot,z)$ is progressively measurable and
	\begin{displaymath}
	\E\bigg[\int_0^T |g(t,z)|^2dt\bigg]<\infty;
	\end{displaymath}
	
	\item[(A2)] There exists some constant $\kappa>0$, such that for any $(\omega,t)\in \Omega\times[0,T]$ and $z,z'\in\mathbb{R}^d$,
	\begin{displaymath}
	|g(\omega,t,z)-g(\omega,t,z')|\leq \kappa|z-z'|;
	\end{displaymath}
	
	\item[(A3)] $g(\omega, t, \cdot)$ is concave for any $(\omega,t)\in\Omega\times [0,T]$;
	
    \item[(A4)] $g(\omega,t,0)=0$ for any $(\omega,t)\in\Omega\times [0,T]$.
\end{description}

The Knightian intertemporal optimal consumption problem under study is defined as follows. Introduce the set
\begin{align*}
\mathcal{X}:=\big\{C\,\big|\,& C \textrm{ is the distribution function of a nonnegative optional random measure on $[0,T]$}\big\},
\end{align*}
let $r=\{r_t\}_{t\in[0,T]}$ be a bounded, progressively measurable process, and $g$ and $h$ be two functions satisfying (A1)-(A4). An agent with initial wealth $w>0$ will pick a consumption plan from the budget feasible set
\begin{equation}
\label{eq:setAhw}
\mathcal{A}_h(w):=\Big\{C\in\mathcal{X}\,|\,\Psi(C):=\widetilde{\mathcal{E}}^h\bigg[\int_0^T \gamma_t\d C_t\bigg]\leq w\Big\},
\end{equation}
where $\gamma_t:=\exp{(-\int_0^t r_sds)}$ and $\widetilde{\mathcal{E}}^h[\,\cdot\,]=-\mathcal{E}^h[\,-\cdot\,]=\mathcal{E}^{\widetilde{h}}[\,\cdot\,]$ (we define $\widetilde{h}(t,z):=-h(t,-z)$). It is easy to check that $\mathcal{A}_h(w)$ is nonempty when $h$ satisfies (A1)-(A4) since $\widetilde{\mathcal{E}}^h[0]=0$. Moreover,
 $\Psi(C)$ can be seen as the minimal initial capital needed to finance a consumption plan $C$ via investing in a financial market.

For a fixed consumption plan $C \in \mathcal{X}$, the agent's level of satisfaction is given by
\begin{equation}
\label{eq:Y}
Y_t^C=\eta_t+\int_0^t \theta_{t,s}\d C_s,
\end{equation}
where $\eta:[0,T]\rightarrow\mathbb{R}_+$ and $\theta:[0,T]^2\rightarrow\mathbb{R}_+$ are continuous. The quantity $\theta_{t,s}$ can be seen as the weight assigned at time $t$ to consumption made at time $s\leq t$ and $\eta_t$ describes the exogenous level of satisfaction that the agent has at time $t$. In \eqref{eq:Y}, and in the following, we interpret the integrals with respect to the optional random measure $\d C$ on $[0,T]$ in the Lebesgue-Stieltjes sense as $\int_0^{\cdot} (\,\cdot\,)\d C_t = \int_{[0,\cdot]} (\,\cdot\,)\d C_t$. In such a way, a possible initial jump of the process $C$ (i.e.\ an initial consumption gulp) is taken into account in the integral.

\begin{remark}
\label{oc17}
Typical $\eta$ and $\theta$ are $\eta_t:=\eta e^{-\beta t}$ and $\theta_{t,s}=\beta e^{-\beta(t-s)}$ with constants $\eta\geq 0,\beta>0$.
\end{remark}

We assume that the agent's utility is of the HHK type and depends on the current level of satisfaction (hence on the past consumption as well) via a certain instantaneous felicity function $u:\Omega \times [0,T]\times\mathbb{R}_+\rightarrow\mathbb{R}$. In particular, we have
\begin{displaymath}
U(C)=\int_0^T u(t,Y_t^C)\d t,
\end{displaymath}
where $u(\omega, \cdot,\cdot)$ is continuous, $u(\omega, t,\cdot)$ is increasing and concave and such that, for any $y \in \R_+$, $(\omega,t) \mapsto u(\omega,t,y)$ is progressively measurable. The goal of the ambiguity's adverse agent is then that of maximizing her expected utility over all budget feasible consumption plans; that is, of finding
\begin{equation}
\label{oc1}
v_{g,h}(w):=\sup_{C\in\mathcal{A}_h(w)}V(C)=\sup_{C\in\mathcal{A}_h(w)}\mathcal{E}^{g}[U(C)].
\end{equation}
In order to simplify notation, in the following we shall omit the subscripts $g,h$ in $v_{g,h}$, as well as its dependence on $w$.



\section{Existence and uniqueness of the optimal consumption plan}
\label{sec:existence}

In this section, we prove existence and uniqueness of a consumption plan solving problem \eqref{oc1}.  For this purpose, we need the following assumption on the budget feasible set.
\begin{description}
	\item[(H1)] The family of budget feasible utilities $\{U(C), C\in\mathcal{A}_h(w)\}$ is uniformly $\P_0$-square-integrable.
\end{description}

Notice that thanks to Assumption (H1) and a priori estimates for BSDEs (cf.\ Proposition 2.1 in  \cite{EPQ}), one has
	\begin{displaymath}
	v =\sup_{C\in\mathcal{A}_h(w)}V(C)<\infty.
	\end{displaymath}

Before moving on with the existence and uniqueness result, we recall the following technical lemma from \cite{BR}.
\begin{lemma}[\cite{BR}, Lemma 2.5]
\label{ocl1}
	\begin{description}
		\item[(i)] There exists some constant $K > 0$, such that for any $C\in\mathcal{X}$ and $t\in[0,T]$,
		\begin{displaymath}
		Y_t^C\leq K(1+C_t);
		\end{displaymath}
		\item[(ii)] If $\{C^n\}_{n=1}^\infty\subset \mathcal{X}$ converges almost surely to $C\in\mathcal{X}$ in the weak topology of measures on $[0,T]$, then we have almost surely $Y_t^{C^n}\rightarrow Y_t^C$ for $t=T$ and for every point of continuity $t$ of $C$.
\end{description}
\end{lemma}

A sufficient condition for Assumption (H1) to hold is given by the next result, whose proof can be found in \ref{appB}.

\begin{lemma}
\label{lem:suffcond}
Suppose that the function $h$ satisfies (A1)-(A4). The family of budget feasible utilities $\{U(C), C\in\mathcal{A}_h(w)\}$ is uniformly $\P_0$-square-integrable if there exists $\alpha\in(0,\frac{1}{2})$ and $K>0$ such that the felicity function $u$ satisfies the power growth condition
	\begin{displaymath}
	\sup_{t\in[0,T]}|u(t,y)|\leq K(1+y^\alpha) \textrm{ for all } y\geq 0.
	\end{displaymath}
\end{lemma}

\begin{theorem}
\label{oct1}
Suppose that the functions $g,h$ satisfy (A1)-(A4). Under Assumption (H1), the utility maximization problem \eqref{oc1} has a solution. Moreover, if $u(\omega,t,\cdot)$ is strictly concave for every $t\in[0,T]$ and $C\rightarrow Y^C$ is injective up to $\P_0$-indistinguishability, such a solution is unique.
\end{theorem}

\begin{proof}
Let $\{C^n\}_{n=1}^\infty\subset \mathcal{A}_h(w)$ such that $\sup_{C\in\mathcal{A}(w)}V(C)=\lim_{n\rightarrow\infty}V(C^n)$. Denote by $\ell$ the convex dual of $h$ and bear in mind Proposition \ref{A1} in Appendix \ref{app}. Since the interest rate is bounded, there exists a constant $K>0$ such that for any $\xi\in D_h$ we have
	\begin{align*}
	\sup_{C\in\mathcal{A}(w)}\E^{\P^\xi}[C_T]&\leq K\sup_{C\in\mathcal{A}(w)}\E^{\P^\xi}\bigg[\int_0^T \gamma_t \d C_t\bigg] \leq K\bigg\{\sup_{C\in\mathcal{A}(w)}\widetilde{\mathcal{E}}^h\bigg[\int_0^T \gamma_t d\ C_t\bigg]+\E^{\P^\xi}\bigg[\int_0^T \ell(s,\xi_s)\d s\bigg]\bigg\}<\infty,
	\end{align*}
	where we have employed \eqref{bound f} in the last inequality. Then, by version of Koml\'{o}s' theorem due to Y.\ Kabanov (Lemma 3.5 in \cite{K99}), there exists a subsequence, for simplicity still denoted by $\{C^n\}$, such that $\P^\xi$-a.s.
	\begin{displaymath}
	\widetilde{C}^n_t:=\frac{1}{n}\sum_{k=1}^{n}C^k_t\rightarrow C^*_t, \ \textrm{as }n\rightarrow\infty
	\end{displaymath}
	for $t=T$ and for every point of continuity $t$ of $C^*$. Since $\P^\xi$ is equivalent to $\P_0$, the above convergence holds $\P_0$-a.s. We claim that $\{\widetilde{C}^n\}$ is also a maximizing sequence for problem \eqref{oc1}. Indeed, the convexity of $\widetilde{\mathcal{E}}^h[\,\cdot\,]$ implies that $\widetilde{C}^n\in \mathcal{A}_h(w)$, for any $n\in\mathbb{N}$. Therefore, we have $V(\widetilde{C}^n)\leq \sup_{C\in\mathcal{A}_h(w)}V(C)$. On the other hand, it is easy to check that
	\begin{displaymath}
	Y^{\widetilde{C}^n}_t=\frac{1}{n}\sum_{k=1}^{n}Y_t^{C^k},
	\end{displaymath}
	and since $u(\omega, t,\cdot)$ and $\mathcal{E}^{g}[\cdot]$ are both concave, we then obtain that
	\begin{displaymath}
	V(\widetilde{C}^n)=\mathcal{E}^{g}[U(\widetilde{C}^n)]\geq \frac{1}{n}\sum_{k=1}^n \mathcal{E}^{g}[U(C^k)]=\frac{1}{n}\sum_{k=1}^n V(C^k).
	\end{displaymath}
	Noting that $\{C^n\}$ is a maximizing sequence, it follows that $\liminf_{n\rightarrow\infty} V(\widetilde{C}^n)\geq \sup_{C\in\mathcal{A}_h(w)}V(C)$, and the claim thus holds.
	
	We then show that $C^*$ is optimal for problem \eqref{oc1}. Since $\gamma$ is continuous, we have by Portmanteau's theorem
	\begin{displaymath}
	\lim_{n\rightarrow\infty}\int_0^T \gamma_t\d\widetilde{C}^n_t=\int_0^T \gamma_t\d C^*_t,\  \P_0\textrm{-a.s.},
	\end{displaymath}
	and by Fatou's lemma and the convexity of $\widetilde{\mathcal{E}}^h[\,\cdot\,]$ it follows that
	\begin{displaymath}
	\widetilde{\mathcal{E}}^h\bigg[\int_0^T \gamma_t\d C^*_t\bigg]\leq \liminf_{n\rightarrow\infty}\widetilde{\mathcal{E}}^h\bigg[\int_0^T \gamma_t\d\widetilde{C}^n_t\bigg]\leq \liminf_{n\rightarrow\infty}\frac{1}{n}\sum_{k=1}^n\widetilde{\mathcal{E}}^h\bigg[\int_0^T \gamma_t\d C^n_t\bigg]\leq w,
	\end{displaymath}
	which implies that $C^*\in\mathcal{A}_h(w)$. Besides, by Lemma \ref{ocl1}, we have $U(\widetilde{C}^n)\rightarrow U(C^*)$, $\P_0$-a.s. Invoking now Assumption (H1) yields that $\E[|U(\widetilde{C}^n)-U(C^*)|^2]\rightarrow 0$, and by estimates for BSDEs (cf.\ Proposition 2.1 in \cite{EPQ}), we obtain that
	\begin{displaymath}
	V(\widetilde{C}^n)=\mathcal{E}^{g}[U(\widetilde{C}^n)]\rightarrow \mathcal{E}^{g}[U(C^*)]=V(C^*).
	\end{displaymath}
	Recalling that $\{\widetilde{C}^n\}$ is a maximizing sequence of problem \eqref{oc1}, it follows that $C^*$ is optimal.
	
	It remains to prove uniqueness. If $C_1$ and $C_2$ are both optimal and they are not indistinguishable, then their associated levels of satisfaction $Y^1=Y^{C^1}$ and $Y^2=Y^{C^2}$ are not indistinguishable neither. By arguments similar to those employed in the proof of Theorem 2.3 in \cite{BR}, on a set with positive probability, $Y^1$ and $Y^2$ differ on an open time interval. By the strict concavity of $u(\omega,t,\cdot)$ and the strict comparison theorem for solutions to BSDEs (cf.\ Therorem 2.2 in \cite{EPQ}), this gives
	\begin{align*}
& V(\frac{1}{2}(C^1+C^2)) =\mathcal{E}^{g}\bigg[\int_0^T u\big(t,\frac{1}{2}(Y^1_t+Y^2_t)\big)\d t\bigg] >\mathcal{E}^{g}\bigg[\int_0^T \frac{1}{2}\big(u(t,Y^1_t)+u(t,Y^2_t)\big)\d t\bigg]\\
	&\geq \frac{1}{2}\bigg(\mathcal{E}^{g}\bigg[\int_0^T u(t,Y^1_t)\d t\bigg]+\mathcal{E}^{g}\bigg[\int_0^T u(t,Y^2_t)\d t\bigg]\bigg) =\frac{1}{2}\big(V(C^1)+V(C^2)\big)=\sup_{C\in\mathcal{A}_h(w)}V(C),
	\end{align*}
	which gives the desired contradiction.
\end{proof}

\begin{remark}
We may check that if $\theta_{t,s}=\theta^1_t\theta^2_s$ for some strictly positive, continuous functions $\theta^1,\theta^2:[0,T]\rightarrow\mathbb{R}^+$, then the mapping $C\mapsto Y^C$ is injective. In particular, if $\theta$ is of the form presented in Remark \ref{oc17}, the injectivity follows. Besides, Theorem \ref{oct1} holds true even if $g$ does not satisfy (A4).
\end{remark}

\begin{remark}\label{general}
We could also consider the utility maximization problem under the general nonlinear expectations
\begin{displaymath}
\label{variational preference}
\begin{split}
&\mathcal{E}[X]:=\inf_{\P\in\mathcal{P}_1}\big(\E^\P[X]+c_1(\P)\big),\quad \text{and} \quad \widetilde{\mathcal{E}}[Y]:=\sup_{\P\in\mathcal{P}_2}\big(\E^\P[Y]-c_2(\P)\big).
\end{split}\end{displaymath}
For this purpose, for any fixed constant $p>1$, we assume that the multiple priors and penalty functions satisfy the following assumptions:		
\begin{description}
\item[(i)]$\sup_{\P\in\mathcal{P}_1}\E\big[\big|\frac{\d\P}{\d\P_0}\big|^p\big]<\infty$, where $p>1$;
		\item[(ii)] $0\leq \inf_{\P\in\mathcal{P}_2}c_2(\P)\leq \sup_{\P\in\mathcal{P}_2}c_2(\P)<\infty$,
\end{description}
and we define the budget feasible set as
\begin{displaymath}
	\widehat{\mathcal{A}}(w):=\Big\{C\in\mathcal{X}\,\big|\, \widetilde{\mathcal{E}}\bigg[\int_0^T \gamma_t \d C_t\bigg]\leq w\Big\}.
\end{displaymath}
Here, $\gamma$ is the discount factor associated with a bounded, progressively measurable interest rate $r$ and $\mathcal{X}$ is the collection of all distribution functions of nonnegative optional random measures as in Section 2. By Assumption  (ii), $\mathcal{A}(w)$ is non-empty for any given initial wealth $w>0$. The level of satisfaction $Y^C$ and the agent's utility $U(C)$ are as in Section 2. The agent aims to maximize her expected utility over all budget feasible consumption plans and the value function is now defined as
\begin{equation}\label{e71}
	\widehat{v}:=\sup_{C\in \widehat{\mathcal{A}}(w)}\mathcal{E}[U(C)].
\end{equation}

Supposing that the family of budget feasible utilities $\{U(C), C\in\widehat{\mathcal{A}}(w)\}$ is uniformly $p^*$-integrable under $\P_0$, where $p^*=p/(p-1)$, we can still show that the utility maximization problem \eqref{e71} has a solution. However, due to the lack of the strict comparison property for $\mathcal{E}[\,\cdot\,]$, we loose uniqueness.
\end{remark}


\section{Sufficient first-order conditions for optimality}
\label{FOCs}

This section is devoted to  the proof of  first-order conditions for optimality. For any two functions $g,h$ satisfying (A1)-(A3), let $f,\ell$ denote their respective convex duals. Now, for any budget feasible consumption plan $C$ of problem \eqref{oc1} set
\begin{align*}
\mathcal{P}_1(C)&:=\Big\{\P^\xi\,\big|\,\xi\in D_g, \mathcal{E}^{g}[U(C)]=\E^{\P^\xi}\bigg[U(C)+\int_0^T f(s,\xi_s)\d s\bigg]\Big\},\\
\mathcal{P}_2(C)&:=\Big\{\P^\xi\,\big|\,\xi\in D_h, \widetilde{\mathcal{E}}^h\bigg[\int_0^T \gamma_t \d C_t\bigg]=\E^{\P^\xi}\bigg[\int_0^T \gamma_t\d C_t-\int_0^T \ell(s,\xi_s)\d s\bigg]\Big\}.
\end{align*}
In fact, in light of Proposition \ref{A1} in Appendix \ref{app}, $\mathcal{P}_1(C)$ can be regarded as the collection of lowest-utility probabilities and $\mathcal{P}_2(C)$ the collection of largest-cost probabilities for the consumption plan $C$. In order to obtain the first-order condition, we need the following additional assumption on the felicity function.

\begin{description}
\item[(H2)] The felicity function $u$ is such that $u(\omega,t,\cdot)$ is strictly concave and differentiable for any $(\omega,t) \in \Omega \times [0,T]$.
\end{description}

\begin{theorem}
\label{oct2}
	Suppose that the functions $g,h$ satisfies Assumptions (A1)-(A4). Under Assumptions (H1)-(H2), a consumption plan $C^*$ solves the utility maximization problem \eqref{oc1} if  there exist some $\P_i:=\P^{\xi^i}\in\mathcal{P}_i(C^*)$, $i=1,2$, such that
	\begin{description}
		\item[(1)] $\displaystyle \widetilde{\mathcal{E}}^h\bigg[\int_0^T \gamma_t\d C^*_t\bigg]=w$;
		\item[(2)] $\displaystyle {\E}^{\P_2}_t\bigg[\frac{\d\P_1}{\d\P_2}\int_t^T \partial_y u(s,Y_s^*)\theta_{s,t}\d s\bigg]\leq M\gamma_t$ for any $t\in[0,T]$ a.s.;
		\item[(3)] $\displaystyle {\E}^{\P_1}\bigg[\int_0^T \Big(\int_t^T \partial_y u(s,Y_s^*)\theta_{s,t}\d s\Big) \d C^*_t\bigg]=M \E^{\P_2}\bigg[\int_0^T \gamma_t\d C^*_t\bigg]$,
	\end{description}
	where $M>0$ is a finite Lagrange multiplier and $Y^*=Y^{C^*}$.
\end{theorem}

\begin{proof}
Assume that $C^*$ satisfies conditions (1)-(3) and consider another budget feasible consumption plan $C\in\mathcal{A}_h(w)$. For simplicity, set $Y:=Y^C$, $Y^*:=Y^{C^*}$, and let
	\begin{align*}
	I&:={\E}^{\P_1}\bigg[\int_0^T \Big(\int_0^s \partial_yu(s,Y_s^*)\theta_{s,t}\d C_t^*\Big) \d s\bigg]={\E}^{\P_1}\bigg[\int_0^T \Big(\int_t^T \partial_yu(s,Y_s^*)\theta_{s,t}\d s\Big) \d C_t^* \bigg],\\
	II&:={\E}^{\P_1}\bigg[\int_0^T \Big(\int_0^s \partial_yu(s,Y_s^*)\theta_{s,t}\d C_t\Big) \d s\bigg]={\E}^{\P_1}\bigg[\int_0^T \Big(\int_t^T \partial_yu(s,Y_s^*)\theta_{s,t}\d s\Big) \d C_t \bigg],
	\end{align*}
	where we have used Fubini theorem. Noting that $\P_2\in\mathcal{P}_2(C^*)$, it is easy to check that
	\begin{align*}
	I&=M \E^{\P_2}\bigg[\int_0^T \gamma_t\d C^*_t\bigg]=M \E^{\P_2}\bigg[\int_0^T \gamma_t\d C^*_t+\int_0^T \ell(s,\xi^2_s)\d s-\int_0^T \ell(s,\xi^2_s)\d s\bigg]\\
	&=M\widetilde{\mathcal{E}}^h\bigg[\int_0^T \gamma_t\d C^*_t\bigg]+M\E^{\P_2}\bigg[\int_0^T \ell(s,\xi^2_s)\d s\bigg]=Mw+M\E^{\P_2}\bigg[\int_0^T \ell(s,\xi^2_s)\d s\bigg],
	\end{align*}
	and
	\begin{align*}
	II=&{\E}^{\P_2}\bigg[\int_0^T \frac{\d\P_1}{\d{\P}_2} \Big(\int_t^T \partial_yu(s,Y_s^*)\theta_{s,t}\d s\Big)\d C_t\bigg]={\E}^{\P_2}\bigg[\int_0^T {\E}^{\P_2}_t\bigg[\frac{\d\P_1}{\d{\P}_2}\int_t^T \partial_y u(s,Y_s^*)\theta_{s,t}\d s\bigg]\d C_t\bigg]\\
	\leq &M\E^{{\P}_2}\bigg[\int_0^T \gamma_t\d C_t+\int_0^T \ell(s,\xi^2_s)\d s-\int_0^T \ell(s,\xi^2_s)\d s\bigg]
	= M \widetilde{\mathcal{E}}^h\bigg[\int_0^T \gamma_t\d C_t\bigg]+M\E^{\P_2}\bigg[\int_0^T \ell(s,\xi^2_s)\d s\bigg]\\
	\leq& Mw+M\E^{\P_2}\bigg[\int_0^T \ell(s,\xi^2_s)\d s\bigg],
	\end{align*}
	where Theorem 1.33 in \cite{J79} has been used in order to obtain the second equality above. Noting that $u(\omega,t,\cdot)$ is strictly concave, we finally have
	\begin{align*}
	V(C^*)-V(C)\geq &\E^{\P_1}\bigg[U(C^*)+\int_0^T f(s,\xi^1_s)\d s\bigg]-\E^{\P_1}\bigg[U(C)+\int_0^T f(s,\xi^1_s)\d s\bigg]\\
	=&\E^{\P_1}\bigg[\int_0^T (u(s,Y^*_s)-u(s,Y_s))\d s\bigg]
	\geq \E^{\P_1}\bigg[\int_0^T \partial_y u(s,Y_s^*)(Y_s^*-Y_s)\d s\bigg]\\
	=&\E^{P_1}\bigg[\int_0^T \Big(\int_0^s \partial_y u(s,Y_s^*)\theta_{s,t}(\d C_t^*-\d C_t)\Big)\d s\bigg] = I-II\geq 0,
	\end{align*}
	which clearly completes the proof.
\end{proof}

\subsection{Some remarks}
\label{sec:remarks}

We conclude this section with some comments.

\begin{remark}
\label{ocl5}
\begin{itemize}
\item[(i)] A careful inspection of the previous proof actually reveals that if $g$ does not satisfy requirement (A4), Theorem \ref{oct2} still holds.

\item[(ii)] Let $\phi_t:={\E}^{\P_1}_t\big[\int_t^T \partial_y u(s,Y_s^*)\theta_{s,t}\d s\big]$. Applying the continuous-time Bayes' rule, condition (2) of Theorem \ref{oct2} is equivalent to
	\begin{displaymath}
	\phi_t\leq M\gamma_t \E^{\P_1}_t\bigg[\frac{\d\P_2}{\d\P_1}\bigg], \ \  t\in[0,T],
	\end{displaymath}
where the state-price density $\gamma_t \E^{\P_1}_t\big[\frac{\d\P_2}{\d\P_1}\big]$ appears on the right-hand side.

\item[(iii)] As a matter of fact, if $C^*$, $\P_1$ and ${\P}_2$ satisfy conditions (2) and (3) of Theorem \ref{oct2}, then for any stopping time $S\leq T$, we have
	\begin{equation}\label{oc3}
	{\E}^{{\P}_2}_S\bigg[\int_S^T\widetilde{\phi}_t \d C^*_t\bigg]=M \E_S^{{\P}_2}\bigg[\int_S^T \gamma_t\d C^*_t\bigg],
	\end{equation}
	where
	\begin{displaymath}
	\widetilde{\phi}_t:=\E^{{\P}_2}_t\bigg[\frac{\d\P_1}{\d{\P}_2}\int_t^T\partial_yu(s,Y_s^*)\theta_{s,t}\d s\bigg],\  t\in[0,T].
	\end{displaymath}
This fact can be indeed proved as follows. Notice that conditions (2) and (3) yield that for any stopping time $S\leq T$,
	\begin{displaymath}
	0=\E^{{\P}_2}\bigg[\int_0^T (\widetilde{\phi}_t-M\gamma_t) \d C_t^*\bigg]\leq \E^{{\P}_2}\bigg[\int_S^T (\widetilde{\phi}_t-M\gamma_t) \d C_t^*\bigg]\leq 0,
	\end{displaymath}
	which clearly implies that $\E^{{\P}_2}\big[\int_S^T (\widetilde{\phi}_t-M\gamma_t) \d C_t^*\big]= 0$. It is then easy to check that
	\begin{displaymath}
	\E^{{\P}_2}_S\bigg[\int_S^T \widetilde{\phi}_t\d C^*_t\bigg]\leq M \E^{{\P}_2}_S\bigg[\int_S^T\gamma_t \d C^*_t\bigg].
	\end{displaymath}
	If now \eqref{oc3} does not hold, defining
$$A:=\Big\{\E^{{\P}_2}_S\bigg[\int_S^T \widetilde{\phi}_t\d C^*_t\bigg]< M \E^{{\P}_2}_S\bigg[\int_S^T\gamma_t \d C^*_t\bigg]\Big\},$$
we have $\P_2(A)>0$ and by the strict comparison property for BSDEs (cf.\ Theorem  2.2 in  \cite{EPQ}) this leads to the contradiction
	\begin{displaymath}
	\E^{{\P}_2}\bigg[\int_S^T (\widetilde{\phi}_t-M\gamma_t) \d C_t^*\bigg]< 0.
	\end{displaymath}
	\end{itemize}
\end{remark}

\begin{remark}
\label{general1}
Consider the general utility maximization problem introduced in Remark \ref{general}. Also for this problem we can establish a set of sufficient conditions for optimality. For any budget feasible consumption plan $C\in\widehat{\mathcal{A}}(w)$, we define
\begin{align*}
	&\mathcal{P}^1(C):=\Big\{\P\in\mathcal{P}_1\,\big|\,\mathcal{E}[U(C)]=\E^\P[U(C)]+c_1(\P)\Big\},\\
	&\mathcal{P}^2(C):=\Big\{\P\in\mathcal{P}_2\,\big|\,\widetilde{\mathcal{E}}\bigg[\int_0^T \gamma_t \d C_t\bigg]=\E^\P\bigg[\int_0^T \gamma_t \d C_t\bigg]-c_2(\P)\Big\},
\end{align*}
and let Assumption (H2) hold. Suppose also that the family of budget feasible utilities $\{U(C), C\in\widehat{\mathcal{A}}(w)\}$ is uniformly $p^*$-integrable under $\P_0$ (where $p^*=p/(p-1)$, for any $p>1$). Then, a consumption plan $C^*$ solves the utility maximization problem \eqref{e71} if there exist some $\P_i\in\mathcal{P}^i(C^*)$, $i=1,2$ such that
	\begin{description}
		\item[(1)] $ \displaystyle \widetilde{\mathcal{E}}\bigg[\int_0^T \gamma_t\d C^*_t\bigg]=w$;
		\item[(2)] $\displaystyle {\E}^{\P_2}_t\bigg[\frac{\d\P_1}{\d\P_2}\Big(\int_t^T \partial_y u(s,Y_s^*)\theta_{s,t}\d s\Big)\bigg]\leq M\gamma_t$ for any $t\in[0,T]$ a.s.;
		\item[(3)] $\displaystyle {\E}^{\P_1}\bigg[\int_0^T\Big(\int_t^T \partial_y u(s,Y_s^*)\theta_{s,t}\d s\Big) \d C^*_t\bigg]=M \E^{\P_2}\bigg[\int_0^T \gamma_t\d C^*_t\bigg]$,
	\end{description}
	where $M>0$ is a finite Lagrange multiplier and $Y^*=Y^{C^*}$.
\end{remark}

\begin{remark}
\label{rem:necessity}
Recall that in the linear case treated in \cite{BR}, the FOCs are also shown to be necessary for optimality (cf. Theorem 3.2 in \cite{BR}). In \cite{BR} the proof of this is organized as follows. First, it is shown that the optimal consumption plan $C^*$ solves an auxiliary problem linearized around $C^*$. Then, a characterization of any solution to such a linearized problem is provided.

In our setting, assuming that
 \begin{itemize}
	\item[(A5)] For each $\omega\in\Omega$, $t\in[0,T]$ and $z\in\mathbb{R}$, the equation $g(\omega,t,z)-xz=f(\omega,t,x)$ admits a unique solution $x\in[-\kappa,\kappa]$, denoted by $x(\omega,t,z)$. Furthermore, $z \mapsto x(\omega,t,z)$ is continuous, for any $(\omega,t) \in\Omega\times [0,T]$,
\end{itemize}
and arguing as in the proof of Lemma B3 in \cite{FLR}, we can show that there exists some $\P_1\in\mathcal{P}_1(C^\star)$ such that the optimal $C^*$ for \eqref{oc1} also solves
\begin{equation}\label{semilinear}
\sup_{C\in\mathcal{A}_h(w)}\E^{\P_1}\left[\int_0^T \E^{\P_1}_t\left[\int_t^T \partial_y u(s,Y^\star_s)\theta_{s,t}\d s\right]\d C_t\right],
\end{equation}
where $Y^*:=Y^{C^*}$. However, the fact that the expectation arising in the set $\mathcal{A}_h(w)$ is nonlinear gives rise to technical difficulties when trying to characterize conditions for the above problem \eqref{semilinear}. Fortunately, we shall see in the next sections that the sufficiency of the FOCs does actually suffice as a verification tool for checking the optimality of a given candidate consumption plan.
\end{remark}


\section{Time-consistency and structure of the optimal consumption plan}
\label{sec:structure}

In this section, we first first study the optimal consumption problem dynamically and prove a version of the dynamic programming principle,  which indicates that if a consumption plan is optimal at time zero, then it is also optimal at any later time. Then, we will show how to construct the optimal consumption plan through an auxiliary backward equation.

\begin{proposition}
\label{ocl6}
	Suppose that the functions $g,h$ satisfy (A1)-(A4). Let $S\leq T$ be a stopping time, $C^*$ be the optimal consumption plan for the utility maximization problem \eqref{oc1}, and set
	\begin{displaymath}
	\mathcal{A}_S(C^*):=\Big\{C\in\mathcal{X}\,\big|\,C|_{[0,S)}\equiv C^*|_{[0,S)}, \Psi_S(C)\leq \Psi_S(C^*)\Big\},
	\end{displaymath}
	where
	\begin{displaymath}
	\Psi_S(C):=\frac{1}{\gamma_S}\widetilde{\mathcal{E}}^g_S\bigg[\int_S^T\gamma_t\d C_t\bigg].
	\end{displaymath}
	Consider then the optimal consumption problem starting at time $S$
	\begin{equation}\label{oc5}
	v_S:=\esssup_{C\in\mathcal{A}_S(C^*)}\mathcal{E}^{h}_S[U(C)],
	\end{equation}
and assume that the felicity function $u$ satisfies $u(\omega,t,0)=0$ for any $(\omega,t)\in\Omega \times [0,T]$. Then the value function $v$ is an $\mathcal{E}^{g}$-supermartingale in the strong sense\footnote{A process $X$ is called an $\mathcal{E}^g$-supermartingale in the strong sense if $X_\tau\in L^2(\mathcal{F}_\tau)$ for any stopping time $\tau$ and for any stopping times $\tau$ and $\sigma$ taking values in $[0,T]$ with $\tau\leq \sigma$, we have $\mathcal{E}^g_\tau[X_\sigma]\leq X_\tau$.}. Besides, $C^*$ is optimal for \eqref{oc5}.
\end{proposition}

\begin{proof}
\emph{Step 1.} Let $S\leq T$ be a stopping time. We first show that the family $\{\mathcal{E}^{g}_S[U(C)], C\in\mathcal{A}_S(C^*)\}$ is upward directed, where $S$ is a stopping time. For any $C^i\in \mathcal{A}_S(C^*)$, $i=1,2$, set $C=C^1 \mathds{1}_A+C^2 \mathds{1}_{A^c}$, where $A:=\{\mathcal{E}^{g}_S[U(C^1)]\geq \mathcal{E}^{g}_S[U(C^2)]\}$ is $\mathcal{F}_S$-measurable. Note that
	\begin{displaymath}
	\Psi_S(C)
	=\frac{1}{\gamma_S}\widetilde{\mathcal{E}}^h_S\bigg[\int_S^T \gamma_r \d C^1_r\bigg]\mathds{1}_A+\frac{1}{\gamma_S}\widetilde{\mathcal{E}}^h_S\bigg[\int_S^T \gamma_r \d C^2_r\bigg]\mathds{1}_{A^c}\leq \Psi_S(C^*),
	\end{displaymath}
	which implies that $C\in \mathcal{A}_S(C^*)$. It is also easy to check that for any $s\in[0,T]$, $Y^C_s=Y^{C^1}_s \mathds{1}_A+Y^{C^2}_s \mathds{1}_{A^c}$. Since $u(s,0)=0$, it follows that $u(s,Y_s^c)=u(s,Y^{C^1}_s) \mathds{1}_A+u(s,Y^{C^2}_s)\mathds{1}_{A^c}$. Therefore, we have
	\begin{align*}
	\mathcal{E}^{g}_S[U(C)]=&\mathcal{E}^{g}_S[U(C^1)\mathds{1}_A+U(C^2)\mathds{1}_{A^c}]=\mathcal{E}^{g}_S[U(C^1)]\mathds{1}_A+\mathcal{E}^{g}_S[U(C^2)]\mathds{1}_{A^c}\\=&\mathcal{E}^{g}_S[U(C^1)]\vee\mathcal{E}^{g}_S[U(C^2)];
	\end{align*}
	that is, the family $\{\mathcal{E}^{g}_S[U(C)], C\in\mathcal{A}_S(C^*)\}$ is upward directed. As a consequence, there exists an increasing sequence $\{\mathcal{E}^{g}_S[U(C^n)]\}_{n=1}^\infty$ such that
	\begin{equation}
	\label{eq:increasing}
	v_S=\lim_{n\rightarrow\infty} \mathcal{E}^{g}_S[U(C^n)],
	\end{equation}
	where $\{C^n\}_{n=1}^\infty\subset \mathcal{A}_S(C^*)$.
\vspace{0.15cm}

\emph{Step 2.} For any stopping times $\tau,\sigma$, with $\tau\leq \sigma$, we have $\mathcal{A}_\sigma(C^*)\subset \mathcal{A}_\tau(C^*)$. Indeed, for any $C\in \mathcal{A}_\sigma(C^*)$, a simple calculation yields that
	\begin{align*}
	\Psi_\tau(C)=&\frac{1}{\gamma_\tau}\widetilde{\mathcal{E}}^h_\tau\bigg[\int_\tau^T \gamma_r\d  C_r\bigg]=\frac{1}{\gamma_\tau}\widetilde{\mathcal{E}}^h_\tau\bigg[\int_\tau^\sigma \gamma_r\d C^*_r+\gamma_\sigma\Psi_\sigma(C)\bigg]\\ \leq &\frac{1}{\gamma_\tau}\widetilde{\mathcal{E}}^h_\tau\bigg[\int_\tau^\sigma \gamma_r\d C^*_r+\gamma_\sigma\Psi_\sigma(C^*)\bigg]=\Psi_\tau(C^*).
	\end{align*}
Now, recalling \eqref{eq:increasing} from Step 1 above, for any $\tau\leq S$, it is easy to check that
	\begin{align*}
	\mathcal{E}^{g}_\tau[v_S]&=\mathcal{E}^{g}_\tau\big[\lim_{n\rightarrow\infty}\mathcal{E}_S^{g}[U(C^n)]\big]
=\lim_{n\rightarrow\infty}\mathcal{E}^{g}_\tau\big[\mathcal{E}^{g}_S[U(C^n)]\big] =\lim_{n\rightarrow\infty}\mathcal{E}_\tau^{g}[U(C^n)]\leq \esssup_{C\in\mathcal{A}_\tau(C^*)}\mathcal{E}^{g}_\tau[U(C)]=v_\tau.
	\end{align*}
\vspace{0.15cm}

\emph{Step 3.} It remains to show that $C^*$ is optimal for problem \eqref{oc5}. Proceeding as in the proof of Theorem \ref{oct1}, there exists a unique consumption plan $\widehat{C}^{S}$ which is optimal for problem \eqref{oc5}. Suppose that $\widehat{C}^{S}$ and $C^*$ are distinguishable on $[S,T]$, so that
\begin{displaymath}
\mathcal{E}^g_S[U(\widehat{C}^{S})]>\mathcal{E}^g_S[U(C^*)].
\end{displaymath}
By invoking the strict comparison theorem for $g$-expectation (cf.\ Theorem 2.2 in \cite{EPQ}) we find
\begin{displaymath}
\mathcal{E}^g[U(\widehat{C}^{S})]=\mathcal{E}^g[\mathcal{E}^g_S[U(\widehat{C}^{S})]]>\mathcal{E}^g[\mathcal{E}^g_S[U(C^*)]]
=\mathcal{E}^g[U(C^*)],
\end{displaymath}
which leads to a contradiction.
\end{proof}

We now move on by studying the structure of the optimal consumption plan. As a matter of fact, Theorem \ref{oct1} indicates that the optimal consumption plan $C^*$ exists, while it does not give an explicit form of $C^*$. Inspired by \cite{BR} and the sufficiency of the first-order conditions for optimality, we now construct $C^*$ through a progressively measurable process $L$, called the minimal level of satisfaction, which is the solution to a backward equation (see \eqref{oc9} below).

In the rest of this section we assume the following dynamics for the level of satisfaction.
\begin{description}
	\item[(H3)] The function $\eta:[0,T]\rightarrow\mathbb{R}$ and $\theta:[0,T]^2\rightarrow \mathbb{R}$ are of following forms:
	\begin{displaymath}
	\eta_t=\eta \exp\bigg(-\int_0^t \beta_s \d s\bigg),\ \ \theta_{t,s}=\beta_s \exp\bigg(-\int_s^t\beta_r\d r\bigg), \ \  0\leq s\leq t\leq T,
	\end{displaymath}
	where $\beta=\{\beta_s\}_{s\in[0,T]}$ is a strictly positive, continuous function and $\eta\geq 0$.
\end{description}

For each fixed $\xi^1\in D_g$, $\xi^2\in D_h$, $M>0$ and stopping time $\tau<T$, consider then the backward equation
\begin{equation}\label{oc9}
\E_\tau\bigg[\int_\tau^T \frac{\d \P^1}{\d \P_0}\bigg|_{\mathcal{F}_t}\partial_yu\bigg(t,\sup_{\tau\leq v\leq t}\bigg\{L_v\exp\bigg(-\int_v^t \beta_s\d s\bigg)\bigg\}\bigg)\theta_{t,\tau}\d t\bigg]=M\gamma_\tau \frac{\d{\P^2}}{\d\P^0}\bigg|_{\mathcal{F}_\tau},
\end{equation}
where $\P^i$ is the probability measure whose Girsanov kernel (with respect to $\P_0$) is given by $\xi^i$, $i=1,2$. By employing Theorem 3 in \cite{BE}, it can be shown that for any $\tau < T$ the above equation admits a unique progressively measurable process $L=L^{M,\P^1,{\P^2}}$ with upper right-continuous paths and such that with $L_T=0$. Starting from this we can then construct two processes by setting
\begin{align*}
&Y^{L^{M,\P^1,{\P^2}}}_t:=\exp\bigg(-\int_0^t \beta_s\d s\bigg)\bigg(\eta\vee\sup_{0\leq v\leq t}\bigg\{L_v^{M,\P^1,{\P^2}}\exp\bigg(\int_0^v \beta_s\d s\bigg)\bigg\}\bigg), \quad t\in[0,T],\\
&C^{L^{M,\P^1,{\P^2}}}_t: =\int_0^t Y^{L^{M,\P^1,{\P^2}}}_s\d s+\int_0^t \beta^{-1}_s\d Y^{L^{M,\P^1,{\P^2}}}_s, \quad t\in[0,T], \quad C^{L^{M,\P^1,{\P^2}}}_{0-}=0.
\end{align*}
According to Lemma 3.9 in \cite{BR}, one has that
\begin{description}
	\item[(i)] $Y^{L^{M,\P^1,{\P^2}}}$ is an adapted RCLL process of bounded variation with $Y^{L^{M,\P^1,{\P^2}}}\geq L^{M,\P^1,{\P^2}}$;
	\item[(ii)] $C^{L^{M,\P^1,{\P^2}}}$ is right-continuous, nondecreasing and adapted. In other words, $C^{L^{M,\P^1,{\P^2}}}\in\mathcal{X}$;
	\item[(iii)] The level of satisfaction induced by $C^{L^{M,\P^1,{\P^2}}}$, denoted by $Y^{C^{L^{M,\P^1,{\P^2}}}}$, coincides with $Y^{L^{M,\P^1,{\P^2}}}$ and is minimal in the following sense:
	\begin{displaymath}
	Y^{C^{L^{M,\P^1,{\P^2}}}}_t=Y^{L^{M,\P^1,{\P^2}}}_t=\inf_{C\in\mathcal{X},Y^C\geq L}Y^C_t, \ \ t\in[0,T].
	\end{displaymath}
\end{description}

Following the terminology of \cite{BR}, in the sequel we shall say that the process $C^L$ constructed above is the consumption plan that tracks the level process $L$.

\begin{theorem}
\label{1}
Recall that $f,\ell$ denote the convex dual of the drivers $g$ and $h$, respectively. Suppose that the functions $g,h$ satisfy (A1)-(A4) and let Assumptions (H1)-(H3) hold. Let also $L^{M,\P^1,\P^2}$ be the solution to \eqref{oc9}. If we can find some $\P^i$ with Girsanov kernel (with respect to $\P_0$) $\xi^i$, $i=1,2$, such that
	\begin{align*}
	\mathcal{E}^{g}\bigg[\int_0^T u\big(t,Y^{L^{M,\P^1,{\P^2}}}_t\big)\d t\bigg]&=\E^{\P^1}\bigg[\int_0^T u\big(t,Y^{L^{M,\P^1,{\P^2}}}_t\big)\d t+\int_0^T f(r,\xi^1_r)\d r\bigg],\\
	\widetilde{\mathcal{E}}^h\bigg[\int_0^T\gamma_t\d C^{L^{M,\P^1,{\P^2}}}_t\bigg]&=\E^{{\P^2}}\bigg[\int_0^T \gamma_t\d C^{L^{M,\P^1,{\P^2}}}_t-\int_0^T \ell(r,\xi^2_r)\d r\bigg],
	\end{align*}
then the consumption plan $C^{L^{M,\P^1,\P^2}}$ is optimal for the utility maximization problem \eqref{oc1} with given initial capital $w=\Psi(C^{L^{M,\P^1,\P^2}})$.
\end{theorem}

\begin{proof}
As the proof follows closely the arguments developed in Theorem 3.13 of \cite{BR}, we omit it in the interest of brevity.
\end{proof}

\begin{remark}
Consider the general utility maximization problem introduced in Remark \ref{general}. Since the sufficiency for optimality still holds as discussed in Remark \ref{general1}, we can then provide the construction for the optimal consumption plan also within such a more general setting. For this purpose, we assume that the dynamics $\eta$ and $\beta$ for the level of satisfaction satisfy (H3). For any constant $M>0$ and any $\P^i\in\mathcal{P}_i$, $i=1,2$, let $L^{M,\P^1,\P^2}$ be the solution to Equation \eqref{oc9} and let $C^{M,\P^1,\P^2}$ be the consumption plan which tracks $L^{M,\P^1,\P^2}$. If $\P^i\in \mathcal{P}^i(C^{M,\P^1,\P^2})$, $i=1,2$, then $C^{M,\P^1,\P^2}$ is optimal for problem \eqref{e71} with initial wealth given by $w=\widetilde{\mathcal{E}}\big[\int_0^T \gamma_t \d C^{M,\P^1,\P^2}_t\big]$.
\end{remark}

\begin{remark}
Existence of the desired minimal level of satisfaction leads to a challenging fixed point problem that we discuss in the following and whose study is left for future research.

Fix the Lagrange multiplier $M$ in Equation \eqref{oc9}. Choose $\xi^{1,1}\in D_g$, $\xi^{2,1}\in D_h$ and let $\P^{i,1}$ be the probability measure whose Girsanov kernel is given by $\xi^{i,1}$, $i=1,2$. Then, there exists a unique solution $L^1$ to Equation \eqref{oc9} with $\P^i=\P^{i,1}$, $i=1,2$. Let $C^1$ be the consumption plan which tracks $L^1$ and $\Gamma(C^1):=\int_0^T \gamma_t \d C_t^1$. Consider then the BSDEs
	\begin{align*}
	&Y^{1,1}_t=U(C^1)+\int_t^T g(s,Z_s^{1,1}) \d s-\int_t^T Z_s^{1,1} \d B_s,\\
	&Y^{2,1}_t=\Gamma(C^1)-\int_t^T h(s,-Z_s^{2,1}) \d s-\int_t^T Z_s^{2,1} \d B_s,
	\end{align*}
	and let $\xi^{1,2}$ and $\xi^{2,2}$ be the solutions to
	\begin{align*}
	&g(s,Z_s^{1,1})-Z_s^{1,1}\xi^{1,2}_s=f(s,\xi^{1,2}_s) \quad \text{and} \quad -h(s,-Z_s^{2,1})-Z_s^{2,1}\xi^{2,2}_s=-\ell(s,\xi^{2,2}_s),
	\end{align*}
	respectively. Then, by Girsanov theorem it is easy to check that
	\begin{align*}
	&\mathcal{E}^g[U(C^1)]=\E^{\P^{1,2}}\bigg[U(C^1)+\int_0^T f(s,\xi^{1,2}_s)\d s\bigg] \quad \text{and} \quad \widetilde{\mathcal{E}}^h[\Gamma(C^1)]=\E^{\P^{2,2}}\bigg[\Gamma(C^1)-\int_0^T \ell(s,\xi^{2,2}_s)\d s\bigg],
	\end{align*}
	where $\P^{i,2}$ is the probability measure with Girsanov kernel $\xi^{i,2}$, $i=1,2$, with respect to $\P_0$.
	
	Defining the mapping $\mathcal{T}:D_g\times D_h\rightarrow D_g\times D_h$ as
	\begin{displaymath}
	\mathcal{T}(\xi^{1,1},\xi^{2,1})=(\xi^{1,2},\xi^{2,2}),
	\end{displaymath}
	we see that if $\mathcal{T}$ has a fixed point $(\xi^1,\xi^2)$, then $C^{\xi^1,\xi^2}$ tracking $L^{\xi^1,\xi^2}$ is optimal for \eqref{oc1}, where $L^{\xi^1,\xi^2}$ is the solution to \eqref{oc9} with $\P^i=\P^{\xi^i}$, $i=1,2$.
\end{remark}


\section{Explicit solution in a stationary homogeneous setting}
\label{sec:explicit solution}

\subsection{Setting and main result}

One can be easily convinced that the sufficient first-order conditions for optimality previously determined still hold when $T=+\infty$. Their proof indeed employs the linear structure of $C \mapsto Y^C$ and the concavity of the instantaneous felicity function, which are clearly not affected by the length of the considered time interval. In this section, we shall use the sufficient optimality conditions in order to provide the explicit solution in an homogeneous setting.

 Consider a financial market with two assets. One of them is a risk-free bond, whose price $S^0$ evolves according to the following equation
\begin{equation}
\label{bond}
\d S^0_t=r S^0_t \d t,
\end{equation}
where $r>0$ is the interest rate. The price for the stock is denoted by $S$ and it satisfies the stochastic differential equation
\begin{equation}
\label{stock}
\d S_t=S^1_t\Big[\mu \d t+ \sigma \d B_t\Big],
\end{equation}
where $\mu$ represents the stock appreciation rates, $\sigma>0$ is the volatility and $B_t$ is a one-dimensional $(\mathcal{F}_t)_t$-Brownian motion. Clearly, there exists a constant $\vartheta\in\mathbb{R}$ such that
 \begin{displaymath}
 \mu-r=\sigma \vartheta.
 \end{displaymath}
 The constant $\vartheta$ is usually referred to as the risk premium.

Imagine now that our agent invests in the financial market and thus selects a portfolio $\pi_t$ at time $t$, where $\pi_t$ is the proportion of her wealth $V_t$ invested in the stock and $\pi^0_t=1- \pi_t$ is the proportion of the wealth invested in the bond. We assume that $\pi$ is predictable, since the agent can only make decisions on the basis of the current amount of available information $\mathcal{F}_t$. The agent can also choose a consumption plan $C \in \mathcal{X}_\infty$, where $C_t$ represents the total amount of consumption made up to time $t$ and
  \begin{align*}
\mathcal{X}_\infty:=\big\{C\,\big|\,& C \textrm{ is the distribution function of a nonnegative optional random measure on $[0,\infty)$}\big\}.
\end{align*}
Also, set $\bar{V}_t:=e^{-rt}V_t$, and suppose that $\limsup_{t\to\infty}\bar{V}_t=0$ a.s.

Let $a,b,a',b'$ be four constants such that $ a'< a<b<b'$. Assume that the risk premia for long and short positions are different and the difference between long and short positions is $a'-a$ (see Example 1.1 in \cite{EPQ}). Then the wealth $V$ associated to the portfolio $\pi$ and consumption plan $C$ evolves as
 \begin{equation}\label{wealth process}
 \d V_t= r V_t \d t +\sigma a' \pi_t V_t \d t+\sigma (a'-a)\pi^-_t V_t \d t +\sigma \pi_t V_t \d B_t- \d C_t,
 \end{equation}
which is clearly equivalent to
\begin{displaymath}
\bar{V}_t= \int_t^\infty \sigma (a\pi^-_s-a' \pi^+_s )\bar{V}_s\d s-\int_t^\infty \sigma \pi_s\bar{V}_s\d B_s +\int_t^\infty e^{-rs}\d C_s.
\end{displaymath}
For any bounded, adapted process $\xi$, set $\varepsilon^\xi_t:=\varepsilon^{1,\xi}_t$, where $\varepsilon^{x_0,\xi}_t:=x_0\exp\big(\int_0^t \xi_s \d B_s-\frac{1}{2}\int_0^t \xi^2_s\d s\big)$ for $x_0>0$. Consequently, we have
\begin{displaymath}
\bar{V}_0=\sup_{\P\in\mathcal{P}^2}\E^\P\left[\int_0^\infty e^{-rs}\d C_s\right],
\end{displaymath}
where \begin{equation}
\label{eq:P2}
\mathcal{P}^2=\Big\{\P^\xi\,\big|\, \xi \textrm{ adapted with values in } [a',a],\, \frac{\d\P^\xi}{\d\P_0}\bigg|_{\mathcal{F}_t}=\varepsilon^\xi_t, 0<t<\infty\Big\},
\end{equation}
which can be interpreted as the set of priors for the cost induced by a consumption plan.

We assume that the felicity function is deterministic and given by
\begin{displaymath}
u(t,y)=e^{-\delta t}\frac{1}{\alpha}y^\alpha,
\end{displaymath}
where $\delta>0$ and $\alpha\in(0,1)$, and $\gamma_t=e^{-rt}$. Moreover, we take a constant interest rate $r>0$ and suppose that the level of satisfaction has dynamics:
\begin{equation}\label{dynamic y}
Y_t^C=\eta e^{-\beta t}+\int_0^t \beta e^{-\beta(t-s)} \d C_s,
\end{equation}
where $\eta,\beta>0$.

Considering now the set of priors for the utility
\begin{equation}
\label{eq:P1}
\mathcal{P}^1:=\Big\{\P^\xi\,\big|\, \xi \textrm{ adapted with values in } [b,b'],\, \frac{\d\P^\xi}{\d\P_0}\bigg|_{\mathcal{F}_t}=\varepsilon^\xi_t, 0<t<\infty\Big\},
\end{equation}
and defining then the nonlinear expectations
\begin{displaymath}
	\mathcal{E}^1[X]:=\inf_{\P\in\mathcal{P}^1}\E^\P[X], \qquad  \mathcal{E}^2[X]:=\sup_{\P\in\mathcal{P}^2}\E^\P[X],
\end{displaymath}
the aim is to solve
\begin{equation}\label{e72}
	\sup_{C\in\mathcal{A}(w)}\mathcal{E}^1\bigg[\int_0^\infty u(t,Y_t^C)\d t\bigg],
\end{equation}
where
$$\mathcal{A}(w):=\Big\{C\in \mathcal{X}_\infty\,\big|\,\mathcal{E}^2\bigg[\int_0^\infty e^{-rt} \d C_t\bigg]\leq w\Big\}.$$

As a matter of fact, we can see from \ref{app} that, by taking $g(z)=b z^+ - b'z^-$ and $h(z)=a' z^+ - az^-$, the representation of $g$-expectation yields
\begin{displaymath}
\mathcal{E}^1[X]=\mathcal{E}^g[X], \quad \text{and} \quad \mathcal{E}^2[X]=\widetilde{\mathcal{E}}^h[X].
\end{displaymath}

In our analysis a crucial role will be played by the backward equation
\begin{equation}
\label{e73}
\E_\tau\bigg[\int_\tau^\infty \beta\exp(-\beta(t-\tau))\varepsilon^{\xi^1}_t \partial_y u\bigg(t,\sup_{\tau\leq v\leq t}\bigg\{L_v \exp(-\beta(t-v))\bigg\}\bigg)\d t\bigg]= M e^{-r\tau} \varepsilon^{\xi^2}_\tau,
\end{equation}
where $\tau$ is any finite stopping time, $M>0$ is a given constant, and $\xi^i$ are the Girsanov kernels (with respect to $\P_0$) such that $\P_i:=\P^{\xi^i}\in\mathcal{P}^i$, $i=1,2$. The solution to \eqref{e73} (if it does exist) is denoted by $L^{M,\P_1,\P_2}$. Moreover, we denote by $C^{{M,\P_1,\P_2}}$ the consumption plan which tracks the level process $L^{M,\P_1,\P_2}$ and by $Y^{{M,\P_1,\P_2}}$ the corresponding level of satisfaction. Similarly to Theorem \ref{1}, we have the following result.

\begin{theorem}\label{t73}
 Suppose that $\P_i\in \mathcal{P}^i(C^{M,\P_1,\P_2})$, $i=1,2$, where for any admissible consumption plan
	\begin{align*}
		&\mathcal{P}^1(C):=\Big\{\P\in\mathcal{P}^1\,\big|\,\mathcal{E}^1\bigg[\int_0^\infty u(t,Y^C_t)\d t\bigg]=\E^\P\bigg[\int_0^\infty u(t,Y^C_t)\d t\bigg]\Big\},\\
		&\mathcal{P}^2(C):=\Big\{\P\in\mathcal{P}^2\,\big|\,\mathcal{E}^2\bigg[\int_0^\infty e^{-rt} \d C_t\bigg]=\E^\P\bigg[\int_0^\infty e^{-rt} \d C_t\bigg]\Big\}.
	\end{align*}
	Then, $C^{M,\P_1,\P_2}$ is an optimal consumption plan for \eqref{e72} with $w=\mathcal{E}^2\big[\int_0^\infty e^{-rt} \d C^{M,\P_1,\P_2}_t\big]$. Consequently, $L^{M,\P_1,\P_2}$ is the optimal minimal level of satisfaction.
\end{theorem}

We can now provide the main result of this section.
\begin{theorem}
\label{t74}
Suppose that $\delta>\alpha r+\frac{\alpha(a-b)^2}{2(1-\alpha)}$. The measures $\P^b$ and $\P^a$ realize the worst-case scenarios for utility and cost of consumption, respectively. Moreover, the optimal minimal level of satisfaction $L^K$ is given by
\begin{displaymath}
L_t^K=\bigg(K e^{(\delta-r)t}\frac{\varepsilon^{a}_t}{\varepsilon^b_t}\bigg)^{\frac{1}{\alpha-1}}=K^{\frac{1}{\alpha-1}}
\exp\bigg\{\frac{a-b}{\alpha-1}B_t-\frac{1}{2(\alpha-1)}[(a^2-b^2)-2(\delta-r)]t\bigg\},
\end{displaymath}
where $K$ is a suitable constant determined by the initial wealth $w$.
Then, the optimal consumption plan is such that
$$C^*_t := C^K_t = \int_0^t e^{-\beta s} \d \overline{C}^*_s,$$
where
\begin{displaymath}
\overline{C}^*_t := \sup_{0 \leq s \leq t}\big(\frac{L_s^K - \eta e^{-\beta s}}{e^{-\beta s}}\Big) \vee 0,\qquad \overline{C}^*_{0-}=0.
\end{displaymath}
Consequently, the optimal level of satisfaction is
$$Y^*_t :=Y^K_t = e^{-\beta t} \big( \eta \vee \sup_{0 \leq s \leq t}(L^K_s e^{\beta s})\big).$$
\end{theorem}

\begin{remark}
\label{rem:deltacond}
We observe that our condition $\delta>\alpha r+\frac{\alpha(a-b)^2}{2(1-\alpha)}$ is consistent with that posed in Theorem 4.7 of \cite{BR}, where
\begin{equation}
\label{BR constraint}
\delta> \alpha r+(1-\alpha)\pi\left(\frac{\alpha\theta'}{1-\alpha}\right)+\alpha\pi(-\theta'),
\end{equation}
with $\pi(\cdot)$ being the Laplace exponent of a L\'{e}vy process $X$ and $\theta'$ the market price of risk. In order to see this, notice that, under a measure $\P^{\xi_1} \in \mathcal{P}^1$, \eqref{e73} rewrites as
\begin{equation}
\label{e73-bis}
\E^{\xi_1}_\tau\bigg[\int_\tau^\infty \beta\exp(-\beta(t-\tau))\partial_y u\bigg(t,\sup_{\tau\leq v\leq t}\bigg\{L_v \exp(-\beta(t-v))\bigg\}\bigg)\d t\bigg]= M e^{-r\tau} \frac{\varepsilon^{\xi^2}_\tau}{\varepsilon^{\xi^1}_\tau}.
\end{equation}
Then, taking $\xi_1\equiv b$ and $\xi_2\equiv a$, the latter is exactly of the form of Equation (17) in \cite{BR}, upon setting (now under $\P^b$)
$$\psi_t = e^{-rt + (a-b)B^b_t - \frac{1}{2}(a-b)^2 t}, \quad t\geq0.$$
Here, $B^b$ is a standard Brownian motion under $\P^b$. In particular, it follows that in our case $\theta'$ of \eqref{BR constraint} is such that $\theta'=a-b$, and simple algebra shows that the right-hand side of \eqref{BR constraint} indeed becomes ours $\alpha r+\frac{\alpha(a-b)^2}{2(1-\alpha)}$.

As a further consequence, we can argue as in the proof of Theorem 4.7 of \cite{BR} and prove that the condition $\delta>\alpha r+\frac{\alpha(a-b)^2}{2(1-\alpha)}$ is in fact also necessary to have a well-posed optimal consumption problem (in the sense that, without that condition one can construct a consumption plan that has finite cost but induces infinite utility). 
\end{remark}


\subsection{On the proof of Theorem \ref{t74}}
\label{proofofthm}

The proof of Theorem \ref{t74} will be organized as follows. First, for any fixed $\P^{\xi^i}\in\mathcal{P}^i$ with $\xi^i$ being a constant, $i=1,2$, under the considered parameters' constellation we solve  backward equation \eqref{e73} explicitly. Then, we verify that $\P^a$ is the largest-cost probability and $P^b$ is the lowest-utility probability for the consumption $C^K$ that tracks $L^K$ defined in Theorem \ref{t74}.

The proof of the next result exploits the time-homogeneity of our setting. It can be found in \ref{appB} for the sake of completeness.
\begin{lemma}
\label{l71}
	For any fixed constants $M$, $\xi^i$ and probability measures $\P^i:=\P^{\xi^i}$, $i=1,2$, the solution to Equation \eqref{e73} is
	\begin{displaymath}
		L_t:=L^{M,\xi^1,\xi^2}_t=\bigg(K e^{(\delta-r)t}\frac{\varepsilon^{\xi^2}_t}{\varepsilon^{\xi^1}_t}\bigg)^{\frac{1}{\alpha-1}},
	\end{displaymath}
	where $K=K(M)$ is the constant satisfying
	\begin{displaymath}
		K\beta \E\bigg[\int_0^\infty e^{-(\delta+\alpha\beta)t}\inf_{0\leq v\leq t}\bigg\{e^{(\delta+\beta(\alpha-1)-r)v}\varepsilon^{\xi^2}_v \frac{\varepsilon^{\xi^1}_t}{\varepsilon^{\xi^1}_v}\bigg\}\d t\bigg]=M.
	\end{displaymath}
\end{lemma}

In order to simplify the notation, set
\begin{displaymath}
\theta:=\frac{b-a}{1-\alpha}, \qquad \lambda:=\frac{1}{2(\alpha-1)}[(a^2-b^2)-2(\delta-r)],
\end{displaymath}
and notice that $\theta>0$.

For any constant $\xi$, let $x^{\alpha}_+(\xi)$ and $x_+(\xi)$ be the largest solutions to the equations $h^{\alpha,\xi}(x)=0$ and $h^\xi(x)=0$, respectively, where
\begin{equation}
\label{e75}
h^{\alpha,\xi}(x):=	\frac{1}{2}\alpha^2\theta^2x^2-\alpha(\lambda-\beta-\theta\xi)x-(\delta+\alpha\beta), \quad x \in \mathbb{R},
\end{equation}
and
\begin{equation}
\label{e76}
h^\xi(x)=\frac{1}{2}\theta^2x^2-(\lambda-\beta-\theta\xi)x-(r+\beta), \quad x \in \mathbb{R}.
\end{equation}
Let $C^K$ be the consumption plan which tracks $L^K$ defined in Theorem \ref{t74} and $Y^K$ be the corresponding level of satisfaction. By the necessary characterization of the optimal consumption plan under a single prior (cf.\ \cite{BR}), the plan $C^K$ is optimal for the linear problem
\begin{equation}\label{linear}
v^{a,b}:=\sup_{C\in \mathcal{A}^a(w)}\E^b\left[\int_0^\infty u(t,Y^C_t)\d t\right],
\end{equation}
where
\begin{equation}
\label{eq:Aaw}
\mathcal{A}^a(w):=\Big\{C\in\mathcal{X}\,\big|\,\E^a\bigg[\int_0^\infty e^{-rt} \d C_t\bigg]\leq w\Big\},
\end{equation}
and $\E^a$ and $\E^b$ are the expectations taken under $\P^a$, $\P^b$, respectively. Besides, proceeding similarly to \cite{BR} (see Equation (38) therein and the following equation for $V(C^K)$) , the expected utility and expected cost associated with consumption plan $C^K$ can be explicitly given by
 \begin{equation}
\label{expected utility}
 \phi^b(\eta):=\E^b\left[\int_0^\infty u(t,Y^K_t)\d t\right]=\frac{1}{\alpha(\delta+\alpha\beta)}\begin{cases}
 \eta^\alpha+\frac{1}{x^\alpha_+(b)-1}K^{\frac{\alpha x_+^\alpha(b)}{\alpha-1}}\eta^{\alpha(1-x_+^\alpha(b))}, &\eta> K^{\frac{1}{\alpha-1}};\\
 \frac{x^\alpha_+(b)}{x^\alpha_+(b)-1}K^{\frac{\alpha}{\alpha-1}}, &\eta\leq K^{\frac{1}{\alpha-1}},
 \end{cases}
 \end{equation}
 and
 \begin{equation}
\label{expected cost}
 \psi^a(\eta):=\E^a\left[\int_0^\infty e^{-rt}\d C^K_t\right]=\frac{1}{\beta}\begin{cases}
 \frac{1}{x_+(a)-1}K^{\frac{x_+(a)}{\alpha-1}}\eta^{1-x_+(a)}, &\eta> K^{\frac{1}{\alpha-1}};\\
 \frac{x_+(a)}{x_+(a)-1}K^{\frac{1}{\alpha-1}}-\eta, &\eta\leq K^{\frac{1}{\alpha-1}}.
 \end{cases}
 \end{equation}

For any $\eta>0$, set
\begin{equation*}
\phi^{\xi^1}(\eta):=\E\bigg[\int_0^\infty \varepsilon^{\xi^1}_t u(t,Y_t^K)\d t\bigg] \quad \text{and} \quad \psi^{\xi^2}(\eta):=\E\bigg[\int_0^\infty e^{-rt}\varepsilon^{\xi^2}_t \d C_t^K\bigg].
\end{equation*}

\begin{lemma}\label{l72}
	Under the same assumptions of Theorem \ref{t74}, we have
	\begin{align*}
		\phi^b(\eta)=\inf_{\xi^1\in [b',b], \textrm{adapted}}\phi^{\xi^1}(\eta), \qquad \psi^a(\eta)=\sup_{\xi^2\in [a,a'], \textrm{adapted}}\psi^{\xi^2}(\eta).
	\end{align*} Besides, the values are finite.
\end{lemma}

\begin{proof}
Notice that $Y^K_t=e^{-\beta t}\{\eta\vee\sup_{0\leq v\leq t}(L^K_ve^{\beta v})\}=e^{-\beta t}\{\eta\vee\sup_{0\leq v\leq t}K^{\frac{1}{\alpha-1}}\exp(\theta B_v-(\lambda-\beta)v)\}$. Applying the Girsanov and Tonelli Theorems, it is easy to check that
\begin{align*}
	\phi^{\xi^1}(\eta)=&\int_0^\infty \frac{1}{\alpha}e^{-(\delta+\alpha\beta)t}\E\bigg[\varepsilon^{\xi^1}_t \Big\{\eta^\alpha\vee\sup_{0\leq v\leq t}K^{\frac{\alpha}{\alpha-1}}\exp(\alpha\theta B_v-\alpha(\lambda-\beta)v)\Big\}\bigg]\d t\\
	=&\int_0^\infty \frac{1}{\alpha}e^{-(\delta+\alpha\beta)t}\E\bigg[ \Big\{\eta^\alpha\vee\sup_{0\leq v\leq t}K^{\frac{\alpha}{\alpha-1}}\exp(\alpha\theta B_v-\alpha(\lambda-\beta)v+\alpha\theta\int_0^v \xi^1_s\d s)\Big\}\bigg]\d t.
\end{align*}
Clearly, for any adapted $\xi^1$ taking values between $b'$ and $b$, we have $\phi^b(\eta)\leq \phi^{\xi^1}(\eta)\leq \phi^{b'}(\eta)$. The assumption $\delta>\alpha r+\frac{\alpha(a-b)^2}{2(1-\alpha)}$ implies that $h^{\alpha,b}(1)<0$. Hence, $x^\alpha_+(b)>1$. Recalling Equation \eqref{expected utility}, $\phi^{b}(\eta)$ is finite.

Then, we calculate $\psi^{\xi^2}(\eta)$. Noting that $\d C_t^K=Y_t^K\d t+\frac{1}{\beta}\d Y_t^K=\frac{1}{\beta}e^{-\beta t}\d (e^{\beta t}Y_t^K)$, a simple calculation yields that
\begin{equation}\begin{split}\label{e77}
	&\int_0^T e^{-rt}\varepsilon^{\xi^2}_t \d C^K_t=\int_0^T \frac{1}{\beta}e^{-(r+\beta)t}\varepsilon^{\xi^2}_t \d (e^{\beta t}Y^K_t)\\
	=&\frac{1}{\beta}e^{-rT}\varepsilon^{\xi^2}_T Y^K_T-\frac{\eta}{\beta}-\frac{1}{\beta}\int_0^T e^{\beta t}Y_t^K\d (e^{-(r+\beta)t}\varepsilon^{\xi^2}_t)\\
	=&\frac{1}{\beta}e^{-rT}\varepsilon^{\xi^2}_T Y^K_T-\frac{\eta}{\beta}+(1+\frac{r}{\beta})\int_0^Te^{-rt}\varepsilon^{\xi^2}_t Y^K_t\d t-\frac{1}{\beta}\int_0^Te^{-rt}\xi^2_t\varepsilon^{\xi^2}_t Y^K_t \d B_t.
\end{split}\end{equation}
Set now $\widetilde{\psi}^{\xi^2}(\eta):=\E\big[\int_0^\infty e^{-rt}\varepsilon^{\xi^2}_t Y^K_t\d t\big]$. Proceeding similarly to the evaluation of $\phi^{\xi^1}(\eta)$, for any adapted process $\xi^2$ taking values in $[a,a']$ we have
\begin{displaymath}
\widetilde{\psi}^{\xi^2}(\eta)=\int_0^\infty e^{-(r+\beta)t}\E\bigg[\eta\vee\sup_{0\leq v\leq t}K^{\frac{1}{\alpha-1}}\exp(\theta B_v-(\lambda-\beta)v+\theta\int_0^v \xi^2_s\d s)\bigg]\d t,
\end{displaymath}
and therefore $\widetilde{\psi}^{a'}(\eta)\leq \widetilde{\psi}^{\xi_2}(\eta)\leq \widetilde{\psi}^{a}(\eta)$.  Calculations analogous to those performed in the proof of Proposition 5.8 in \cite{FLR} imply that, for $\xi^2\in \{a,a'\}$,
 \begin{displaymath}
 \widetilde{\psi}^{\xi^2}(\eta)=\frac{1}{\beta+r}\begin{cases}
 \frac{1}{x_+(\xi^2)-1}K^{\frac{x_+(\xi^2)}{\alpha-1}}\eta^{1-x_+(\xi^2)}+\eta, &\eta> K^{\frac{1}{\alpha-1}};\\
 \frac{x_+(\xi^2)}{x_+(\xi^2)-1}K^{\frac{1}{\alpha-1}}, &\eta\leq K^{\frac{1}{\alpha-1}}.
 \end{cases}
 \end{displaymath}
The assumption $\delta>\alpha r+\frac{\alpha(a-b)^2}{2(1-\alpha)}$ yields that $h^a(1)<0$ and $h^{a'}(1)<0$. Therefore, we have $x_+(a)>1$ and $x_+(a')>1$. Hence, the quantity $\widetilde{\psi}^{\xi^2}(\eta)$ is positive and finite for any adapted $\xi^2\in[a,a']$. Since by following the steps of the proof of Lemma 4.9 in \cite{BR} we can show that
\begin{displaymath}
	\lim_{T\rightarrow \infty}e^{-rT}\varepsilon^{\xi^2}_T Y^K_T=0,
\end{displaymath}
taking expectation on both sides of \eqref{e77}, and letting $T$ to infinity, we obtain
\begin{align*}
	\psi^{\xi^2}(\eta)=(1+\frac{r}{\beta})\widetilde{\psi}^{\xi^2}(\eta)-\frac{\eta}{\beta}.
\end{align*}
It is now straightforward to see that $\psi^{a'}(\eta)\leq \psi^{\xi^2}(\eta)\leq \psi^a(\eta)$, which then completes the proof.
\end{proof}

The results collected so far finally allow us to prove Theorem \ref{t74}.
\vspace{0.15cm}

\begin{proof}[Proof of Theorem \ref{t74}]
By Lemma \ref{l71} and \ref{l72}, we know that $L^K$ is the solution to \eqref{e73} with $\P_1=\P^b$, $\P_2=\P^a$ and  $\P^b\in \mathcal{P}^1(C^K)$, $\P^a\in \mathcal{P}^2(C^K)$. By Theorem \ref{t73}, it remains to find an appropriate constant $K$, such that $w=\psi^{a}(\eta)$,  where $\psi^a(\eta)$ is given by \eqref{expected cost}. By simple calculations it can be shown that the needed $K$ is given by
\begin{displaymath}
K=\begin{cases}
\left(\frac{x_+(a)-1}{{x_+(a)}}(\beta w+\eta)\right)^{\alpha-1}, & w\geq \frac{\eta}{\beta(x_+(a)-1)};\\
(\beta(x_+(a)-1)\eta^{x_+(a)-1}w)^{\frac{\alpha-1}{x_+(a)}}, &\textrm{otherwise}.
\end{cases}
\end{displaymath}
\end{proof}


\subsection{On the portfolio financing the optimal consumption plan}
\label{sec:portfolio}

In this section, we will find the portfolio process $\pi$ needed in order to finance the optimal consumption plan derived in the previous section. For this purpose, let $V^a_t$, $V_t$ be the present values of the future consumption $C^K$ taken under the probability $\P^a$ and the set of multiple priors $\mathcal{P}^2$, respectively; i.e.,
\begin{align*}
&V^a_t:=\E^a_t\left[\int_t^\infty e^{-r(s-t)}\d C^K_s\right]=\E_t\left[\int_t^\infty \frac{\varepsilon^a_s}{\varepsilon_t^a}e^{-r(s-t)}\d C^K_s\right],\\
&V_t:=\esssup_{\P^\xi\in\mathcal{P}^2}\E^\xi_t\left[\int_t^\infty e^{-r(s-t)}\d C^K_s\right]=\esssup_{\P^\xi\in\mathcal{P}^2}\E_t\left[\int_t^\infty \frac{\varepsilon^\xi_s}{\varepsilon_t^\xi}e^{-r(s-t)}\d C^K_s\right].
\end{align*}

\begin{lemma}
\label{verification}
Recall $\psi^a$ as in \eqref{expected cost}. One has
\begin{displaymath}
V_t=V^a_t=e^{\theta B_t-\lambda t}\psi^a(e^{-\theta B_t+\lambda t}Y_t^K).
\end{displaymath}
\end{lemma}

\begin{proof}
By Equation \eqref{e77}, we have
\begin{displaymath}
V_t^a=\frac{(\beta+r)e^{rt}}{\beta\varepsilon^a_t}\E_t\left[\int_t^\infty e^{-rs}\varepsilon^a_s Y^K_s\d s\right]-\frac{Y^K_t}{\beta}:=\widetilde{V}^a_t-\frac{Y^K_t}{\beta}.
\end{displaymath}
On the other hand, the Markov property implies that
\begin{align*}
\widetilde{V}_t^a=&\frac{(\beta+r)e^{rt}}{\beta\varepsilon^a_t}\E_t\left[\int_t^\infty e^{-(r+\beta)s}\varepsilon^a_s\left\{\eta\vee \sup_{0\leq v\leq s}K^{\frac{1}{\alpha-1}}\exp(\theta B_v-(\lambda-\beta)v)\right\}\d s\right]\\
=&\frac{(\beta+r)}{\beta}\E_t\left[\int_0^\infty e^{-(r+\beta)s}\frac{\varepsilon^a_{s+t}}{\varepsilon_t^a}\left\{Y^K_t\vee e^{-\beta t}\sup_{t\leq v\leq s+t}K^{\frac{1}{\alpha-1}}\exp(\theta B_v-(\lambda-\beta)v)\right\}\d s\right]\\
=&\frac{(\beta+r)}{\beta}\E_t\left[\int_0^\infty e^{-(r+\beta)s}\frac{\varepsilon^a_{s+t}}{\varepsilon_t^a}\left\{Y^K_t\vee e^{-\lambda t}\sup_{0\leq v\leq s}K^{\frac{1}{\alpha-1}}\exp(\theta B_{v+t}-(\lambda-\beta)v)\right\}\d s\right]\\
=&\frac{(\beta+r)}{\beta}e^{\theta B_t-\lambda t}\E_t\left[\int_0^\infty e^{-(r+\beta)s}\frac{\varepsilon^a_{s+t}}{\varepsilon_t^a}\left\{\eta\vee \sup_{0\leq v\leq s}K^{\frac{1}{\alpha-1}}\exp(\theta (B_{v+t}-B_t)-(\lambda-\beta)v)\right\}\d s\right]\bigg|_{\eta=\widetilde{Y}^K_t}\\
=&\frac{(\beta+r)}{\beta}e^{\theta B_t-\lambda t}\E\left[\int_0^\infty e^{-(r+\beta)s}\varepsilon_s^a\left\{\eta\vee \sup_{0\leq v\leq s}K^{\frac{1}{\alpha-1}}\exp(\theta B_v-(\lambda-\beta)v)\right\}\d s\right]\bigg|_{\eta=\widetilde{Y}^K_t}\\
=&e^{\theta B_t-\lambda t}\left(\psi^a(e^{-\theta B_t+\lambda t}Y^K_t)+\frac{1}{\beta}e^{-\theta B_t+\lambda t}Y^K_t\right),
\end{align*}
where $\widetilde{Y}^K_t=e^{-\theta B_t+\lambda t}Y^K_t$. All the above analysis yields that $V^a_t=e^{\theta B_t-\lambda t}\psi^a(e^{-\theta B_t+\lambda t}Y_t^K)$. Then, following the arguments developed in the proof of Theorem 5.11 in \cite{FLR} we conclude that $V_t=V_t^a$.
\end{proof}

\begin{proposition}\label{portfolio}
Under the same assumptions of Theorem \ref{t74}, we have
\begin{equation}
\label{eq:optimalpi}
\pi_t=\frac{\theta x_+(a)}{\sigma} =:\pi \quad \text{for all}\,\, t \geq 0.
\end{equation}
\end{proposition}

\begin{proof}
It is easy to check that for any $t\geq 0$, $e^{-\theta B_t+\lambda t}Y^K_t\geq K^{\frac{1}{\alpha-1}}$. By Lemma \ref{verification} and \eqref{expected cost}, applying It\^{o}'s formula, we find
\begin{displaymath}
dV_t=\theta x_+(a) V_t \d B_t+ \textrm{terms of bounded variation}.
\end{displaymath}
Compared with \eqref{wealth process}, the latter gives that the portfolio is actually constant and such that $\pi_t=\theta x_+(a)/\sigma$ for all $t\geq0$.
\end{proof}

\begin{remark}
By Theorem \ref{t74} and the analysis in Section 4 of \cite{BR}, the optimal policy $C^K$ for the utility maximization problem \eqref{e72} is also optimal for the linear problem \eqref{linear}. Hence, the portfolio process \eqref{eq:optimalpi} is also such to finance the consumption plan appearing in \eqref{eq:Aaw} and thus is consistent with that found in Theorem 4.14 of \cite{BR}.
\end{remark}

Recall that $x_+(a)$ is the largest solution to the equation $h^a(x)=0$, where $h^a$ is defined in \eqref{e76}. We now study the dependence of the portfolio \eqref{eq:optimalpi} with respect to $b-a$ (measuring somehow the discrepancy between the beliefs of the agents related to the utility and cost from consumption), the volatility $\sigma$, and the parameter of relative risk-aversion $1-\alpha$.

\begin{proposition}
\label{comparative statics}
Recall $\pi$ from \eqref{eq:optimalpi}. On has that:
\begin{enumerate}
\item $\pi$ is decreasing with respect to $\sigma$;
\item $\pi$ is decreasing with respect to $1-\alpha$;
\item Let $\hat{\delta}:=\beta + \frac{\delta - r}{\alpha-1}$. Then:
\begin{itemize}
\item[(i)] If $\hat{\delta}\geq 0$, then $\pi$ is increasing with respect to $b-a$.

\item[(ii)] If $\hat{\delta}<0$ and $\delta-2\alpha\delta+\alpha^2 r+\beta\alpha(1-\alpha)\leq 0$, then then $\pi$ is decreasing with respect to $b-a$.

\item[(iii)] If $\hat{\delta}<0$ and $\delta-2\alpha\delta+\alpha^2 r+\beta\alpha(1-\alpha)> 0$, then $\pi$ is decreasing with respect to $b-a$ when $(b-a)^2\in(0, - 2(1-\alpha)\hat{\delta})$ and increasing when $(b-a)^2\in(-2(1-\alpha)\hat{\delta}, (1-\alpha)\frac{2(\delta-\alpha r)}{\alpha})$.
\end{itemize}
\end{enumerate}
\end{proposition}

\begin{proof}
Since it readily follows that $\pi$ is decreasing with respect to $\sigma$, we move directly on by proving the second claim of the proposition.

Setting
 \begin{displaymath}
 \widetilde{\delta}:=(a-b)^2+2(\delta-r), \qquad  I(\alpha):=\frac{\beta}{(a-b)}(1-\alpha)+\frac{\widetilde{\delta}}{2(b-a)},
 \end{displaymath}
 it is easy to check that
 \begin{displaymath}
 \theta x_+(a)=I(\alpha)+\sqrt{I^2(\alpha)+2(r+\beta)}=:H(\alpha),
 \end{displaymath}
 and
 \begin{displaymath}
 H'(\alpha)=\left(1+\frac{I(\alpha)}{\sqrt{I^2(\alpha)+2(r+\beta)}}\right)I'(\alpha)\geq 0.
 \end{displaymath}
From the latter, it follows that $\pi$ is increasing in $\alpha$, and this proves the second claim.

As for the third claim, recall that $\theta=(b-a)/(1-\alpha)$ and notice that
\begin{displaymath}
\theta x_+(a)=\frac{1-\alpha}{2}\theta-\frac{\hat{\delta}}{\theta}+
\sqrt{(\frac{1-\alpha}{2}\theta-\frac{\hat{\delta}}{\theta})^2+2(r+\beta)}:=F(\theta).
\end{displaymath}
Defining $G(\theta):=\frac{1-\alpha}{2}\theta-\frac{\hat{\delta}}{\theta}$, simple calculations imply that
\begin{displaymath}
F'(\theta)=\frac{\sqrt{G^2(\theta)+2(r+\beta)}+G(\theta)}{\sqrt{G^2(\theta)+2(r+\beta)}}G'(\theta),
\end{displaymath}
and $G'(\theta)=(1-\alpha)/2+\hat{\delta}\theta^{-2}$. Since $F'\geq0$ if and only if $G'\geq0$, it follows that the monotonicity of $\pi$ with respect to $b-a$ is the same as the monotonicity of $G$ with respect to $\theta$. Now, the requirement $\delta>\alpha r+\frac{\alpha(a-b)^2}{2(1-\alpha)}$ implies that $0<\theta^2<\frac{2(\delta-\alpha r)}{\alpha(1-\alpha)}$. We then have the following three cases, completing the proof.

\textbf{Case (i).} Suppose that $\hat{\delta}\geq 0$. Then, we have $G'(\theta)\geq 0$ for any $\theta^2\in(0,\frac{2(\delta-\alpha r)}{\alpha(1-\alpha)})$.

\textbf{Case (ii).} Suppose that $\hat{\delta}<0$ and $\delta-2\alpha\delta+\alpha^2 r+\beta\alpha(1-\alpha)\leq 0$. It is easy to check that for any $\theta^2\in(0,\frac{2(\delta-\alpha r)}{\alpha(1-\alpha)})$, $G'(\theta)\leq 0$.

\textbf{Case (iii).} Suppose that $\hat{\delta}<0$ and $\delta-2\alpha\delta+\alpha^2 r+\beta\alpha(1-\alpha)> 0$. By simple calculation, we have $G'(\theta)< 0$ when $\theta^2\in(0,\frac{2\hat{\delta}}{\alpha-1})$ and $G'(\theta)>0$ when $\theta^2\in(\frac{2\hat{\delta}}{\alpha-1},\frac{2(\delta-\alpha r)}{\alpha(1-\alpha)})$.
\end{proof}

Let us comment on the comparative statics of the optimal portfolio. First, we obtain the intuitive result that an increase in volatility $\sigma$  reduces investment in the risky asset. More risk averse agents also invest less in the risky assets than their less risk-averse counterparts.
These results are well in line with the usual comparative statics in asset pricing models.

Knightian uncertainty and market friction are well described by the difference of the parameters
  $b-a$ in our model because these parameters describe the respective worst case models for the agent's minimal expected utility and maximal expected cost. Interestingly, it is not always the case that an increase of this   term, i.e. a better expected return for investing in the risky asset,  leads to more exposure in the risky asset.
  We have to distinguish two cases here. When $\hat{\delta}=\beta + \frac{\delta - r}{\alpha-1} \ge 0$, the  optimal level of satisfaction diverges to infinity (compare also the discussion in \cite{BR00}   in the frictionless case).  In this case, the optimal portfolio is indeed decreasing in market friction. In the opposite case, we have $(1-\alpha) \beta < \delta - r $.  When interest rates $r$ are quite low, impatience $\delta$ is high, and the depreciation rate $\beta$ is low, i.e. the agent enjoys past consumption for a long time, the agent optimally consumes a lot in early years and thus does not invest for higher  future consumption as he enjoys his past consumption for a long time and as he is relatively impatient. He is thus not interested in the long-run prospective of higher returns and thus even decreases his investment in the risky asset if the returns increase.


\subsection{Abstention from the asset marekt}
\label{sec:priors}

We have already seen in Remark \ref{rem:deltacond} that the requirement on $\delta$ imposed in the statement of Theorem \ref{t74} is necessary in order to have a well-posed optimization problem. We now want to better understand the role of the other assumption made within this section; namely, our requirement $-\kappa\leq a'< a < b <b'\leq \kappa$ imposed on the parameters defining the sets of priors $\mathcal{P}_1$ and $\mathcal{P}_2$ (cf.\ \eqref{eq:P1} and \eqref{eq:P2}).

\begin{proposition}
Suppose that there exists some probability $\P$ such that $\P\in \mathcal{P}^1\cap \mathcal{P}^2$ and let $\delta > \alpha r$\footnote{This condition arise from ours $\delta > \alpha r + \frac{\alpha(a-b)^2}{2(1-\alpha)}$ needed to avoid an ill-posed problem. Indeed, in the case $\mathcal{P}^1\cap \mathcal{P}^2 \neq \emptyset$, we can simply take $a=b$, thus leading to $\delta > \alpha r$.}. Then, the optimal consumption plan for the utility maximization problem \eqref{e72} is deterministic.
\end{proposition}

\begin{proof}
It is easy to check that $L^K_v=(Ke^{(\delta-r)v})^{\frac{1}{\alpha-1}}$ is the solution to \eqref{e73} with $\xi^1=\xi^2$ such that $\P^{\xi^1}=\P^{\xi^2}=\P$ and Lagrange multiplier $M=\frac{K\beta}{r+\beta}$.  Clearly, the consumption plan that tracks such $L^K$ is a deterministic function and hence, the utility and the cost induced by this consumption plan are deterministic as well. Therefore, we have
	 \begin{align*}
	 &\inf_{\P\in\mathcal{P}^1}\E^\P\bigg[\int_0^\infty  u(t,Y_t^K)\d t\bigg]=\int_0^\infty  u(t,Y_t^K)\d t=\E^\P\bigg[\int_0^\infty  u(t,Y_t^K)\d t\bigg] ,\\
	 &\sup_{P\in\mathcal{P}^2}\E^\P\bigg[\int_0^\infty  e^{-rt} \d C^K_t\bigg]=\int_0^\infty  e^{-rt} \d C^K_t=\E^\P\bigg[\int_0^\infty  e^{-rt} \d C^K_t\bigg],
	 \end{align*}
	 which implies that $\P\in \mathcal{P}^1(C^K)\cap \mathcal{P}^2(C^K)$. Hence, $\P$ is indeed the worst-case scenario for the utility and cost associated to the deterministic plan $C^K$.

In the following, we complete the proof by determining the appropriate constant $K$ such that the induced cost equals to the initial wealth $w$.

\textbf{Case 1.} Suppose that $\delta<r+(1-\alpha)\beta$. In this case, $Y^K_t=(e^{-\beta t}\eta)\vee (K^{\frac{1}{\alpha-1}}e^{\frac{\delta-r}{\alpha-1}t})=(e^{-\beta t}\eta)\vee L^K_t$.
It remains to find $K$ such that the cost $\int_0^\infty e^{-rt} \d C_t^K$ equals to $w$, where
	 \begin{displaymath}
	 \int_0^\infty e^{-rt}\d C_t^K=\begin{cases} K^{\frac{1}{\alpha-1}}
\frac{(1-\alpha)(\beta+r)}{\beta(\delta-\alpha r)}-\frac{\eta}{\beta}, & \eta \leq K^{\frac{1}{\alpha-1}};\\
\frac{(1-\alpha)\beta+r-\delta}{\beta (\delta-\alpha r)}\eta^{\frac{\alpha r-\delta}{(1-\alpha)\beta+r-\delta}}K^{-\frac{r+\beta}{(1-\alpha)\beta+r-\delta}}, & \eta > K^{\frac{1}{\alpha-1}}.
\end{cases}
	 \end{displaymath}
Simple calculations imply that
\begin{displaymath}
K=\begin{cases}
\left[\frac{(\eta+\beta w)}{(1-\alpha)(\beta+r)}\right]^{\alpha-1}, &w\geq\frac{\beta\eta(1-\alpha)+\eta(r-\delta)}{\beta(\delta-\alpha r)}; \\
\left[\frac{\beta w(\delta-\alpha r)}{(1-\alpha)\beta+r-\delta}\right]^{-\frac{(1-\alpha)\beta+r-\delta}{r+\beta}}\eta^{\frac{\delta-\alpha r}{\delta+r}}, &\textrm{otherwise}.
\end{cases}
\end{displaymath}

\textbf{Case 2.} Suppose that $\delta\geq r+(1-\alpha)\beta$. In this case, $Y^K_t=e^{-\beta t}(\eta\vee L_0^K)=e^{-\beta t}(\eta\vee K^{\frac{1}{\alpha-1}})$. Recalling the dynamics of level of satisfaction \eqref{dynamic y}, it is easy to check that $K=(\eta+\beta w)^{\alpha-1}$ and the optimal consumption $C^*$ should be such that $C^*_t=C^*_0=w$. That is, the agent will consume all his initial wealth at the original time.
\end{proof}


\section*{Acknowledgments}
\noindent Financial support by the German Research Foundation (DFG) through the Collaborative Research Centre 1283 ``Taming uncertainty and profiting from randomness and low regularity in analysis, stochastics and their applications'' is gratefully acknowledged by the authors.


\appendix
\renewcommand\thesection{Appendix A}
\section{ }
\label{app}
\renewcommand\thesection{A}

In this appendix we introduce the $g$-expectation and we provide its important properties.

Consider a filtered probability space $(\Omega,\mathcal{F}_T,(\mathcal{F}_t)_{t\in[0,T]},\P_0)$ satisfying the usual conditions of right-continuity and completeness and on which it is defined a $d$-dimensional Brownian motion $B=\{B_t\}_{t\in[0,T]}$. For any terminal value $X\in L^2(\mathcal{F}_T)$, the collection of all $\mathcal{F}_T$-measurable and square-integrable random variables, consider the following BSDE
\begin{displaymath}
Y_t^X=X+\int_t^T g(s,Z_s^X)\d s-\int_t^T Z^X_s\d B_s, \quad t \leq T.
\end{displaymath}

By the results in \cite{PP90}, under Assumptions (A1)-(A2) there exists a unique pair of solution $(Y^X,Z^X)$. We define the $g$-conditional expectation for $X$ as
\begin{displaymath}
\mathcal{E}^g_{t,T}[X]:=Y^X_t.
\end{displaymath}
For simplicity, we denote $\mathcal{E}_{0,T}^g[X]$ by $\mathcal{E}^g[X]$. The $g$-expectation coincides with a variational preference in the following sense (cf.\ \cite{EPQ}).

\begin{proposition}
\label{A1}
	Suppose that $g$ satisfies (A1)-(A3). For each $\omega\in\Omega$, $t\in[0,T]$ and $\theta\in\mathbb{R}^d$, let
	\begin{displaymath}
	f(\omega,t,\theta):=\sup_{z\in\mathbb{R}^d}(g(\omega,t,z)-z\cdot\theta)
	\end{displaymath}
	be the convex dual of $g$. Denote by $D_g$ be the collection of all progressively measurable processes $\{\xi_t\}_{t\in[0,T]}$ such that
	\begin{displaymath}
	\E\bigg[\int_0^T |f(s,\xi_s)|^2ds\bigg]<\infty.
	\end{displaymath}

Let $\tau$ be a stopping time satisfying $0\leq t\leq\tau\leq T$. For each $\mathcal{F}_\tau$-measurable and square integrable random variable $X$, we have the following representation
	\begin{displaymath}
	\mathcal{E}_{t,\tau}^g[X]=\essinf_{\xi\in D_g}\{\E^{\P^\xi}_t[X]+\alpha_{t,\tau}(\xi)\},
	\end{displaymath}
	where the probability measure $\P^\xi$ is defined on $(\Omega, \mathcal{F}_T)$ through
	\begin{displaymath}
	\frac{\d \P^\xi}{\d \P_0}:=\exp\Big(\int_0^T \xi_s dB_s-\frac{1}{2}\int_0^T \xi_s^2ds\Big),
	\end{displaymath}
 $\E^{\P^\xi}_t[\,\cdot\,]$ is the expectation under $\P^{\xi}$ conditioned on $\mathcal{F}_t$, and the penalty function is defined as
	\begin{displaymath}
	\alpha_{t,\tau}(\xi):=\E^{\P^\xi}_t\bigg[\int_t^\tau f(s,\xi_s)\d s\bigg].
	\end{displaymath}
\end{proposition}

One of the most important properties of the classical conditional expectation is time-consistency, i.e.\ the so-called tower property. In fact, this property still holds for the $g$-conditional expectations. More precisely, we have the following proposition, whose details can be found in \cite{P97} and \cite{CHMP}.
\begin{proposition}\label{a21}
	Suppose that $g$ satisfies (A1)-(A3). The conditional $g$-expectation satisfies the following properties:
	\begin{description}
		\item[(1)] \textbf{Strict comparison:} if $X\leq Y$, then $\mathcal{E}^g_{t,T}[X]\leq \mathcal{E}_{t,T}^g[Y]$. Furthermore, if $\P_0(X<Y)>0$, then $\mathcal{E}^g_{t,T}[X]< \mathcal{E}_{t,T}^g[Y]$;

		\item[(2)] \textbf{Time consistency:} for any $0\leq s\leq t\leq T$, $\mathcal{E}^g_{s,T}[\mathcal{E}^g_{t,T}[X]]=\mathcal{E}^g_{s,T}[X]$;

		\item[(3)] \textbf{Concavity:} $\mathcal{E}_{t,T}^g[\,\cdot\,]$ is concave; i.e., for any $X,Y\in L^2(\mathcal{F}_T)$ and $\lambda\in[0,1]$, we have $\mathcal{E}_{t,T}^g[\lambda X+(1-\lambda)Y]\geq \lambda\mathcal{E}_{t,T}^g[X]+(1-\lambda)\mathcal{E}_{t,T}^g[Y]$;
		
		\item[(4)] \textbf{Fatou's lemma:}  Suppose that for any $n\in\mathbb{N}$, $\mathcal{E}^g[X_n]$ exists and $X_n\geq X$ (respectively, $X_n\leq X$), where $X\in L^2(\mathcal{F}_T)$. Then, we have
		$$
		\liminf_{n\rightarrow \infty}\mathcal{E}^g[X_n]\geq \mathcal{E}^g[\liminf_{n\rightarrow \infty} X_n] \quad (\text{respectively,}\,\, \limsup_{n\rightarrow \infty}\mathcal{E}^g[X_n]\leq \mathcal{E}^g[\limsup_{n\rightarrow \infty} X_n]).
		$$
	\end{description}
\end{proposition}

If we assume, additionally, that the function $g$ satisfies (A4), it is easy to check that for any $X\in L^2(\mathcal{F}_{T_1})\subset L^2(\mathcal{F}_{T_2})$, where $T_1\leq T_2$, we also have
\begin{displaymath}
\mathcal{E}^g_{t,T_1}[X]=\mathcal{E}_{t,T_2}^g[X].
\end{displaymath}
In this case, we denote $\mathcal{E}^g_{t,T}[X]$ by $\mathcal{E}^g_{t}[X]$. The advantage of using $g$ which satisfies also condition (A4) lies in the fact that it preserves almost all properties as the classical expectation, with the exception of linearity.

\begin{proposition}\label{a22}
	Suppose that $g$ satisfies (A1), (A2) and (A4). The conditional $g$-expectation satisfies the following:
	\begin{description}
\item[(1)] \textbf{Translation invariance:} if $Z\in L^2(\mathcal{F}_t)$, then for all $X\in L^2(\mathcal{F}_T)$, $\mathcal{E}^g_t[X+Z]=\mathcal{E}^g_t[X]+Z$;

		\item[(2)] \textbf{Local property:} for an event $A\in\mathcal{F}_t$, we have $\mathcal{E}^g_t[X \mathds{1}_A+Y \mathds{1}_{A^c}]=\mathcal{E}^g_t[X]\mathds{1}_A+\mathcal{E}^g_t[Y]\mathds{1}_{A^c}$;
		
		\item[(3)] \textbf{Constant preserving:} if $X\in L^2(\mathcal{F}_t)$, we have $\mathcal{E}^g_t[X]=X$.
	\end{description}
\end{proposition}


\renewcommand\thesection{Appendix B}
\section{ }
\label{appB}
\renewcommand\thesection{B}

In this appendix we provide the proof of some technical result needed in the paper.

\begin{lemma}
\label{l1}
	For any bounded progressively measurable processes $\xi$ and $\sigma$, and for any constant $p>0$, the random variable $\exp\big(p(\int_0^T \xi_s\d s+\int_0^T \sigma_s\d B_s)\big)$ is $\P_0$-integrable.
\end{lemma}
\emph{Proof.}
It is easy to check that
\begin{align*}
&\E\bigg[\exp\bigg(p\big(\int_0^T\sigma_t \d B_t+\int_0^T\xi_t \d t\big)\bigg)\bigg]\\
=&\E\bigg[\exp\bigg(\int_0^T p\sigma_t \d B_t-\int_0^T p^2\sigma^2_t\d t\bigg)\exp\bigg(\int_0^Tp^2\sigma^2_t\d t+\int_0^Tp\xi_t \d t\bigg)\bigg]\\
\leq &\bigg(\E\bigg[\exp\big(\int_0^T 2p\sigma_t \d B_t-\frac{1}{2}\int_0^T (2p\sigma_t)^2\d t\big)\bigg]\bigg)^{1/2}\bigg(\E\bigg[\exp\big(\int_0^T2p^2\sigma^2_t\d t+\int_0^T2p\xi_t \d t\big)\bigg]\bigg)^{1/2}\\
=&\bigg(\E\bigg[\exp\big(\int_0^T2p^2\sigma^2_t\d t+\int_0^T2p\xi_t \d t\big)\bigg]\bigg)^{1/2}\leq \text{const.},
\end{align*}
where we use the fact that $\{\exp\big(\int_0^t 2p\sigma_s \d B_s-\frac{1}{2}\int_0^t (2p\sigma_s)^2\d s\big)\}_{t\in[0,T]}$ is a martingale by Novikov's condition. The proof is then complete.
\vspace{0.25cm}

\emph{Proof of Lemma \ref{lem:suffcond}}
\vspace{0.15cm}

By the power growth condition and Lemma \ref{ocl1}, it is easy to check that
	\begin{displaymath}
	U(0)\leq U(C)\leq K\int_0^T (1+|Y_t^C|^\alpha)\d t\leq K(1+C_T^\alpha).
	\end{displaymath}
	Therefore, it suffices to show that the family $\{C_T^\alpha, C\in\mathcal{A}_h(w)\}$ is uniformly $\P_0$-square-integrable. To accomplish that recall Proposition \ref{A1} in \ref{app} (in particular see the definition of the set $D_h$), and for any $p> 2$ with $\alpha p <1$, by the H\"{o}lder inequality, we have
	\begin{equation}
	\label{eq:CT}
	\E[C_T^{\alpha p}]\leq \E\bigg[C_T \frac{\d \P^\xi}{\d \P_0}\bigg]^{\alpha p}\E\bigg[\bigg(\frac{\d \P^\xi}{\d \P_0}\bigg)^{\frac{-\alpha p}{1-\alpha p}}\bigg]^{1-\alpha p},
	\end{equation}
	where $\xi\in D_h$ and $|\xi|\leq \kappa$. By Lemma \ref{l1}, $\E\bigg[\bigg(\frac{\d \P^\xi}{\d \P_0}\bigg)^{\frac{-\alpha p}{1-\alpha p}}\bigg]\leq \text{const.}$, and letting $\ell$ be the convex dual of $h$, by Proposition \ref{A1} for some $K>0$ we have
	\begin{align}
	\label{eq:CT2}
	\E\bigg[C_T \frac{\d \P^\xi}{\d \P_0}\bigg]&=\E^{\P^\xi}[C_T]\leq K \E^{\P^\xi}\bigg[\int_0^T \gamma_t \d C_t\bigg] \leq K\bigg(\widetilde{\mathcal{E}}^h\bigg[\int_0^T \gamma_t\d C_t\bigg]+\E^{\P^\xi}\bigg[\int_0^T \ell(s,\xi_s)\d s\bigg]\bigg),
	\end{align}
where, in order to obtain the first inequality, we have used that the interest rate is bounded.

Also, noticing that $\xi\in D_h$, by Lemma \ref{l1}, it follows that
    \begin{equation}\label{bound f}
    \E^{\P^\xi}\bigg[\int_0^T \ell(s,\xi_s)\d s\bigg]=\E\bigg[\frac{\d\P^\xi}{\d\P_0}\int_0^T \ell(s,\xi_s)\d s\bigg]\leq \bigg(\E\bigg[\bigg(\frac{\d\P^\xi}{\d\P_0}\bigg)^2\bigg]\bigg)^{1/2}\bigg(T\E\bigg[\int_0^T \ell^2(s,\xi_s)\d s\bigg]\bigg)^{1/2}<\infty.
    \end{equation}
Feeding the latter back into \eqref{eq:CT2} and using \eqref{eq:CT} we conclude that the family $\{C^\alpha_T,C\in\mathcal{A}_h(w)\}$ is $p$-integrable under $\P_0$, thus leading to the desired result.
\vspace{0.25cm}

\emph{Proof of Lemma \ref{l71}}
\vspace{0.15cm}

By the strong Markov property and a change of variable, we have
	\begin{align*}
		&\E_\tau\bigg[\int_\tau^\infty \varepsilon^{\xi^1}_t \partial_yu\bigg(t,\sup_{\tau\leq v\leq t}\bigg\{L_v \exp\big(-\int_v^t\beta_sds\big)\bigg\}\bigg)\theta_{t,\tau}\d t\bigg]\\
		=&\E_\tau\bigg[\int_\tau^\infty \varepsilon^{\xi^1}_t e^{-\delta t}\beta e^{-\beta(t-\tau)}\inf_{\tau\leq v\leq t}\bigg\{Ke^{(\delta-r)v}e^{-\beta(\alpha-1)(t-v)}\frac{\varepsilon^{\xi^2}_v}{\varepsilon^{\xi^1}_v}\bigg\}\d t\bigg]\\
		=&K\beta e^{-r\tau} \E_\tau\bigg[\int_0^\infty e^{-(\delta+\alpha\beta)t} \inf_{0\leq v\leq t}\bigg\{e^{(\delta+\beta(\alpha-1)-r)v}\varepsilon^{\xi^2}_{v+\tau}\frac{\varepsilon^{\xi^1}_{t+\tau}}
{\varepsilon^{\xi^1}_{v+\tau}}\bigg\}\d t\bigg]\\
		=&K\beta e^{-r\tau} \E\bigg[\int_0^\infty e^{-(\delta+\alpha\beta)t} \inf_{0\leq v\leq t}\bigg\{e^{(\delta+\beta(\alpha-1)-r)v}\varepsilon^{x_2,\xi^2}_{v}
\frac{\varepsilon^{x_1,\xi^1}_{t}}{\varepsilon^{x_1,\xi^1}_{v}}\bigg\}\d t\bigg]\bigg|_{x_1=\varepsilon^{\xi^1}_\tau,x_2=\varepsilon^{\xi^2}_\tau}\\
		=&K\beta e^{-r\tau} \E\bigg[\int_0^\infty e^{-(\delta+\alpha\beta)t} \inf_{0\leq v\leq t}\bigg\{e^{(\delta+\beta(\alpha-1)-r)v}\varepsilon^{\xi^2}_{v}\frac{\varepsilon^{\xi^1}_{t}}{\varepsilon^{\xi^1}_{v}}\bigg\}\d t\bigg]\varepsilon^{\xi^2}_\tau.
	\end{align*}
	Hence, the result follows.


\end{document}